\documentclass[12pt, reqno]{amsart}

\usepackage{tikz}
\usepackage{hvfloat}

\usepackage{stmaryrd}
\usetikzlibrary{matrix,arrows,decorations.pathmorphing}
\usepackage{amsfonts}
\usepackage{graphicx}
\usepackage{mathpazo}
\usepackage{euler}
\usepackage{amssymb}
\usepackage{amsmath,mathtools}
\usepackage{fullpage}
\usepackage{caption}
\usepackage{amsthm}
\usepackage{thmtools}
\usepackage{enumerate}
\usepackage{picinpar} % for pictures in paragraphs
\usepackage{tikz-cd}
\usepackage{color}
\usepackage{latexsym}
\usepackage{enumitem} 
\usepackage[toc,page]{appendix}
\usepackage{url} 
\usepackage{multicol}
\usepackage{rotating}
\usepackage{todonotes}
\usepackage[numbers]{natbib}         % bibliography style
\usepackage[colorlinks]{hyperref} 
\hypersetup{colorlinks=true, linkcolor=black, citecolor = red, urlcolor = black}
\makeatletter
\newcommand{\customlabel}[2]{%
   \protected@write \@auxout {}{\string \newlabel {#1}{{#2}{\thepage}{#2}{#1}{}} }%
   \hypertarget{#1}{#2}}
\makeatother

\usepackage{setspace}
\usepackage{listings}
\usepackage{comment}
\usepackage{longtable}

\usepackage{tabularx}

\newtheorem{thmx}{Theorem}

\numberwithin{equation}{subsection}
\newtheorem{theorem}[subsection]{Theorem}
\newtheorem{lemma}[subsection]{Lemma}

\newtheorem{coro}[subsection]{Corollary}
\newtheorem{conjecture}[subsection]{Conjecture}

\newtheorem{prop}[subsection]{Proposition}

\theoremstyle{definition}
\newtheorem{defn}[subsection]{Definition}

\theoremstyle{remark}
\newtheorem{remark}[subsection]{Remark}

\newcommand{\mZ}{\mathbb{Z}}
\newcommand{\mR}{\mathbb{R}}
\newcommand{\mQ}{\mathbb{Q}}
\newcommand{\mF}{\mathbb{F}}
\newcommand{\mC}{\mathbb{C}}

\newcommand{\mP}{\mathbb{P}}

\newcommand{\mE}{\mathbb{E}}

\newcommand{\rra}{\rightarrow}

\newcommand{\ZZ}[1]{\mathbb{Z} / #1 \mathbb{Z}}
\newcommand{\brk}[1]{ \left\lbrace #1 \right\rbrace }
\newcommand{\inv}{^{-1}}

\newcommand{\mif}{\text{ if }}
\newcommand{\tth}{^{\text{th}}}

\newcommand{\iso}{\cong}

\newcommand{\ext}{\hookrightarrow}
\newcommand{\pwr}[1]{\left( #1 \right)}

\newcommand{\lrra}{\longrightarrow}

\newcommand{\wt}[1]{\widetilde{ #1 }}

\def\quotient#1#2{\raise1ex\hbox{$#1$}{\Large/} \lower1ex\hbox{$#2$}}

\DeclareMathOperator{\Aut}{Aut}

\DeclareMathOperator{\id}{id}

\DeclareMathOperator{\GL}{GL}

\DeclareMathOperator{\SL}{SL}

\DeclareMathOperator{\Mat}{M}
\DeclareMathOperator{\Gal}{Gal}

\DeclareMathOperator{\Jac}{Jac}

\DeclareMathOperator{\USp}{USp}

\DeclareMathOperator{\End}{End}
\DeclareMathOperator{\Sp}{Sp}
\DeclareMathOperator{\SU}{SU}
\DeclareMathOperator{\ST}{ST}
\DeclareMathOperator{\Tw}{Tw}
\DeclareMathOperator{\GSp}{GSp}
\DeclareMathOperator{\Lefs}{L}
\DeclareMathOperator{\AST}{AST}

\newcommand\bq{{\mathbb Q}}

\newcommand{\cdef}[1]{{\color{black} \textbf{#1}}}

\newcommand{\abGal}{G_{\mQ}}
\newcommand{\cocyc}{\xi}

\begin{document}
\title{The twisting Sato\--Tate group of the curve $y^2 = x^{8}  - 14x^4 + 1$}

\author{Sonny Arora}
\address{Department of Mathematics, The Pennsylvania State University, University Park, PA 16802}
\email{sza149@psu.edu}
\thanks{The first author was partially supported by National Science Foundation grant DMS-1056703.}

\author{Victoria Cantoral-Farf\'an}
\address{Institut de Math\'ematiques de Jussieu - Paris Rive Gauche (IMJ-PRG)
UP7D - B\^atiment Sophie Germain - 75205 Paris France}
\email{victoria.cantoral-farfan@imj-prg.fr}
\thanks{The second author was supported by a Conacyt fellowship.}

\author{Aaron Landesman}
\address{Department of Mathematics, Stanford University, \mbox{Stanford, CA 94305}}
\email{aaronlandesman@stanford.edu}

\author{Davide Lombardo}
\address{D\'epartement de Math\'ematiques d'Orsay, Universit\'e Paris-Sud, 91400 Orsay (France)}
\email{davide.lombardo@math.u-psud.fr}

\author{Jackson S. Morrow}
\address{Department of Mathematics and Computer Science, Emory University,
Atlanta, GA 30322}
\email{jmorrow4692@gmail.com}

\begin{abstract}

We determine the twisting Sato\--Tate group
of the genus $3$ hyperelliptic curve $y^2 = x^{8}  - 14x^4 + 1$ and show that all possible subgroups of the twisting Sato\--Tate group arise as the Sato\--Tate group of an explicit twist of $y^2 = x^{8}  - 14x^4 + 1$. Furthermore, we prove the generalized Sato\--Tate conjecture for the Jacobians of all $\mathbb Q$-twists of the curve $y^2 = x^{8}  - 14x^4 + 1$.
\end{abstract}
\date{\today}
\maketitle

\section{Introduction}

In this paper, we prove the generalized Sato\--Tate conjecture for all $\mQ$-twists
of the genus $3$ hyperelliptic curve $y^2 = x^8 - 14 x^4 + 1$, which corresponds to an isolated point in the moduli space of genus 3 hyperelliptic curves due to its large geometric automorphism group. In doing so, we exhibit a $\mQ$-twist of this curve with its full automorphism group defined over $\bq$, and to the authors' knowledge, this is the first example of a curve over $\mQ$ with Sato\--Tate distribution, the distribution of the normalized traces of Frobenius, given by the measure
$$\frac{1}{6\pi} \int_a^b \sqrt{4-\left( \frac{t}{3} \right)^2} dt$$
for all intervals $[a,b] \subset [-6,6]$.

We begin in \autoref{subsection:generalized-sato-tate-conjecture} and
\autoref{subsection:progress-on-generalized-sato-tate} with some background on the generalized Sato\--Tate conjecture. 
We then describe the special properties of the genus $3$ curve $y^2 = x^8 - 14 x^4 + 1$ in \autoref{subsection:an-isolated-point},
and conclude the introduction by elaborating on our main results in \autoref{subsection:main-results}.

\subsection{The generalized Sato\--Tate conjecture}
\label{subsection:generalized-sato-tate-conjecture}
Let $A$ be an abelian variety of dimension $g$ defined over $\mQ$. The generalized \cdef{Sato\--Tate conjecture} predicts that the Haar measure of a certain compact group $G$, with $G \subset \USp(2g),$ governs the distribution of the normalized Euler factors $\overline{L}_{p}(A,T)$, as $p$ varies over the primes of good reduction of $A$. The normalized Euler factor at a prime $p$ is the polynomial $\overline{L}_{p}(A,T) = L_{p}(A,T/p^{1/2})$, where $L_{p}(A,T)$ is the $L$-polynomial of $A$ at $p$.
Define $\alpha_i$ so that $L_p(A,T) = \prod_{i=1}^{2g}(1-\alpha_iT)$.
Recall that the $L$-polynomial has the defining property that for each positive integer $n$,
$$\#A(\mF_{p^n}) = \prod_{i=1}^{2g}(1-\alpha_i^n).$$

We now explain what we mean when we say $G$ ``governs" the distribution of the $L$-polynomials. Serre \cite{serre2011lectures} has defined, using $\ell$-adic monodromy groups, a compact real Lie subgroup of $\USp(2g)$ called the \cdef{Sato\--Tate group} of $A$, denoted $\ST(A)$, satisfying the following condition: for each $p$ at which $A$ has good reduction, there exists a conjugacy class of $\ST(A)$ whose characteristic polynomial equals $\overline{L}_{p}(A,T) := \sum_{i=0}^{2g}a_i(A)(p)T^i$.
For a brief summary of this construction, see \cite[Section~2.1]{fite2012sato}. For $i=0,1,\dots , 2g$, let $I_i$ denote the interval $\left[-{2g \choose i},{2g \choose i}\right]$, and consider the map
$$\Phi_i\colon \ST(A) \subset \USp(2g) \lrra I_i \subset \mR$$
that sends an element of $\ST(A)$ to the $i\tth$ coefficient of its characteristic polynomial. Let $\mu(\ST(A))$ denote the Haar measure of $\ST(A)$ and let $\Phi_{i\, *}(\mu(\ST(A)))$ denote the pushforward of the Haar measure on $I_i$. We now state the generalized Sato\--Tate conjecture.\footnote{The generalized Sato\--Tate conjecture naturally extends to abelian varieties defined over number fields $K$, however we shall only be concerned with the case of $K = \mQ$.}

\begin{conjecture}[Generalized Sato\--Tate]\label{conj:ST}
For each $i = 0,1,\dots , 2g$, the $a_i(A)(p)$'s are equidistributed, with respect to increasing size, on $I_i$ with respect to $\Phi_{i\, *}(\mu(\ST(A)))$.
\end{conjecture}

\subsection{Progress on proving the generalized Sato\--Tate conjecture}
\label{subsection:progress-on-generalized-sato-tate}
The classical Sato\--Tate conjecture is concerned with the case that $A$ is an elliptic curve over $\mathbb Q$ without CM. Here, $g=1$ and $\ST(A) \iso \USp(2) \iso \SU(2)$. This form of the conjecture was recently proved in culmination of a project several years in the making; see \cite[Example~8.1.5.2]{serre2011lectures} for a complete list of references. For elliptic curves $E$ with CM over a general number
field $K$, there are two cases. If $E$ has CM defined over $K$, then $\ST(E) \iso \text{U}(1)$ where $\text{U}(1) \ext \SU(2)$ via $u\mapsto \begin{psmallmatrix} u & 0 \\ 0 & \overline{u} \end{psmallmatrix}$. If $E$ has CM not defined over $K$, then $\ST(E)$ is isomorphic to the normalizer of $\text{U}(1)$ in $\SU(2)$. 

The recent work of Fit\'e, Kedlaya, Rotger, and Sutherland has presented an explicit version of the generalized Sato-Tate conjecture for abelian surfaces over number fields $K$. The classification in \cite[Table~8]{fite2012sato} gives a description in the form of explicit equations of curves whose Jacobians realized each of the 52 Sato\--Tate groups that can and do arise for abelian surfaces. The authors of \cite{fite2012sato} also show that only 34 of these arise for Jacobians of genus 2 curves defined over $\mQ$. In \cite{fite2014sato}, the authors prove that 18 of these 34 subgroups can be realized as the Sato\--Tate group of a $\mQ$-twist of either the curve $y^2 = x^5 - x$ or $y^2 = x^6 + 1$ and that the generalized Sato\--Tate conjecture holds for the Jacobians of $\mQ$-twists of the aforementioned curves.

For higher dimensional abelian varieties, there have been a few sporadic results which we briefly recall. Fit\'e and Sutherland \cite{fite2014sato3} provide effective algorithms to compute the traces of Frobenius at primes of good reduction for the curves $y^2  = x^8 + c$ and $y^2 = x^7 - cx$ and determine the Sato\--Tate groups that arise generically and for specific $c\in\mQ^*$. In her thesis \cite{lorenzo2014}, Lorenzo, in joint work with Fit\'e and Sutherland, computed the Sato\--Tate groups and Sato\--Tate distributions for $\mQ$-twists of the Fermat and Klein quartics and proved Conjecture \ref{conj:ST} for the Jacobians of these curves. Another result worth mentioning is \cite{fite2012Satomotives}, where the authors establish a group-theoretic classification of Sato\--Tate groups of self-dual motives of weight 3 with rational coefficients and prescribed Hodge numbers.

In this article, we continue the investigation of Sato\--Tate in genus 3 by computing the Sato\--Tate groups of the Jacobians of $\mQ$-twists of the genus 3 curve $y^2 = x^{8}  - 14x^4 + 1$ and proving the generalized Sato\--Tate conjecture for such twists.

\subsubsection*{Notation} Throughout, let $\overline{\mQ}$ denote a fixed algebraic closure of $\mQ$, and let $G_{\mQ} := \Gal(\overline{\mQ}/\mQ)$ denote the absolute Galois group of $\mQ$. For any algebraic variety $X$ defined over $\mQ$ and any extension $L/\mQ$, we use $X_L$ to denote the algebraic variety obtained over $L$ by the base change of $\mQ\ext L$. For abelian varieties $A$ and $B$ defined over $\mQ$, we write $A\sim_L B$ (resp.~$A\sim B$) to indicate that there is an isogeny between $A$ and $B$ defined over $L$ (resp.~$\mQ$).

\subsection{An isolated point in the moduli space of genus 3 hyperelliptic curves}
\label{subsection:an-isolated-point}
Let $C_0$ denote the hyperelliptic curve $y^2 = x^{8}  - 14x^4 + 1$ defined over $\mQ$.  In this paper, we first compute the twisting Sato\--Tate group $\ST_{\Tw}(C_0)$, which is a compact Lie group with the property that the Sato\--Tate group of the Jacobian of any twist of $C_0$ is isomorphic to a subgroup of $\ST_{\Tw}(C_0)$. 

In \cite[Table~1]{gutierrez2012hyperelliptic}, the authors determine the automorphism groups of genus 3 hyperelliptic curves. The curve $C_0$ is particularly fascinating, because the geometric automorphism group of $C_0$ is isomorphic to $S_4 \times \ZZ{2}$, and, up to $\overline{\mQ}$-isomorphism, $C_0$ is the unique genus 3 hyperelliptic curve with $\#\Aut((C_0)_{\overline{\mQ}}) = 48$. 
%By \cite[Table~1]{gutierrez2012hyperelliptic}, we have that there are no genus 3 hyperelliptic curves with automorphism group of size more than $48$. 
Further, we have $\Jac C_0 \sim_{\overline{\mQ}} E^3 $, where $E$ is the non-CM elliptic curve $y^2=x^4-14x^2+1$. \texttt{Magma} computes that $\Aut(C_0) \iso (\ZZ{2})^3$ and that all of the automorphisms of $C_0$ are defined over $\mQ(\zeta_8)$. 

For computational reasons, we shall primarily work with the $\mQ(\zeta_8)$-twist $C$ of $C_0$ defined by $y^2 = x^{8} + 14x^4 + 1$. We compute that $\Aut(C) \iso D_8$ and that all of the automorphisms of $C$ are defined over $\mQ(i)$. In section \ref{section:the-right-twist}, we shall determine a twist $C'$ of $C$ with its full geometric automorphism group defined over $\mQ$. For convenience, we record the defining equations for these three twists:
\begin{align}
	\label{equation:curve-notation}
C_0 \colon y^2 &= x^8 - 14x^4 + 1 
& C \colon y^2 &= x^8 + 14x^4 + 1 
 & C' \colon  & \begin{cases}
x^2+y^2+z^2=0 \\
-2t^2=x^4+y^4+z^4,
\end{cases} 
\end{align}
where the latter curve $C'$ lives in the weighted projective space $\mP_{\mQ}(1,1,1,2)$ with variables $x,y,z$ of weight $1$ and $t$ of weight 2.

\begin{remark}
The family of $\mQ$-twists of $C$ is interesting, not only due to the extremal, geometric automorphism group, but also because they are twists of the modular curve $X_0(48)$; this can be checked with the \texttt{Magma} intrinsic
\begin{verbatim}
SimplifiedModel(SmallModularCurve(48));
\end{verbatim}
By \cite{bruin2015hyperelliptic}, $X_0(48)$ is one of two hyperelliptic modular curves whose hyperelliptic involution does not correspond to the Atkin\--Lehner involution, and via \textit{loc.~cit.~}Theorem~6, the points on $X_0(48)$ give rise to infinitely many isogeny classes, with maximal size, of elliptic curves over quadratic fields.
\end{remark}

\subsection{Main results}
\label{subsection:main-results}
In this paper, we show that all possible
subgroups of the twisting Sato\--Tate group of $C_0$ (see~\autoref{definition:twisting-sato-tate-group})
arise as the Sato\--Tate groups of twists of $C_0$.
We also give explicit equations for a representative
curve with every such subgroup; see Table \ref{Table4}. Using this classification, we prove the generalized Sato\--Tate conjecture for all $\mQ$-twists of $C_0$.

\begin{thmx}\label{thm:main}
Let $\wt{C}$ be a $\mQ$-twist of $C_0$.
\begin{enumerate}
	\item Up to conjugacy in $\USp(6)$, there are exactly 9 distinct possibilities for the Sato\--Tate group $\ST(\Jac \wt{C})$, and an explicit twist $\wt{C}$ realizing each Sato\--Tate group is given in Table \ref{Table4}.
\item For $i = 0,\dots , 6$, the $a_i(\wt{C})(p)$ are equidistributed on $I_i = \left[-{6 \choose i}, {6 \choose i}\right]$ with respect to a measure $\mu(a_i(\wt{C}))$ that is uniquely determined by the Sato\--Tate group $\ST(\Jac (\wt{C}))$.
\item Conjecture \ref{conj:ST} holds for $\wt{C}$.
\end{enumerate}
\end{thmx}
We prove Theorem \ref{thm:main}(1) in section \ref{sec:CompGrp} after we compute the twisting Sato\--Tate group of $C$ in Lemma \ref{lemma:algebraic-sato-tate-group-computation}. In Propositions \ref{prop:MomentSequences2} and \ref{prop:MomentSequences4}, we compute pushforwards of measures determined by the distinct Sato\--Tate groups and establish Theorem \ref{thm:main}(2). Finally, in Theorem \ref{thm:ProofST}, we complete the proof of Theorem \ref{thm:main}(3).

\begin{remark}
In this paper, we give an example of a curve, namely $C'$ from~\eqref{equation:curve-notation}, over $\mQ$, whose Sato\--Tate distribution is given by the measure
$$\frac{1}{6\pi} \int_a^b \sqrt{4-\left( \frac{t}{3} \right)^2} dt,$$
where $\left[ a,b \right] \subset \left[ -6,6 \right]$. 
Although examples of a curve over $\mathbb Q$ with such a distribution were likely known to the experts,
to the authors' knowledge, all previous in print examples 
of curves with such a distribution were defined over 
an extension of $\mathbb Q$, such as $\mQ(\sqrt{-3})$. 
\end{remark}

\begin{remark}[Finding the right twist]
An important step in our proof is exhibiting the twist $C'$ of $C_0$, which has key property that $\Jac C' \sim E_1 \times E_2 \times E_3$ where $E_i$ are non-CM elliptic curves over $\mQ$. In \cite[Section~4]{fite2012sato}, the authors introduce twists of the genus 2 curves $y^2 = x^5 - x$ and $y^2  = x^6 + 1$ whose Jacobians are also isogenous to a product of elliptic curves; however, \cite[Section~4]{fite2012sato} does not discuss the methods for obtaining such twists. We initially tried to find $C'$ using a brute force search
over all genus 3 hyperelliptic curves $X$ of the form $y^2 = f(x)$, where $f$ is a degree 8 symmetric polynomial satisfying $\#\Aut(X) = 48$. 

After several overnight computations with no fruitful results, we employed a different approach. Recall that the twists of a curve $C$ are in bijection with the cohomology classes in $H^1(G_{\mQ}, \Aut (C_{\overline{\mQ}}))$. We first identify our desired twist as a cocycle in $H^1(G_{\mQ}, \Aut (C_{\overline{\mQ}}))$ and then exploit the fact that $H^1(G_{\mQ}, \Aut (C_{\overline{\mQ}}))$ also parameterizes automorphisms of $C$-torsors. In particular, we can identify the function field of our desired twist using invariant theory on $\overline{\mQ}(C)$, and whence we can find a model for $C'$; see section~\ref{section:the-right-twist} for details. We found that any $C'$ with the desired property is hyperelliptic after base change to $\overline \mQ$, but $C'$ necessarily does not have a degree 2 map to $\mP^1_{\mathbb Q}$ defined over $\mQ$; in particular, our initial brute force search was destined to failure.
\end{remark}

\subsection{Organization}\label{sub:Org}
The outline of this paper is as follows:
In \autoref{section:background}, we recall the definition of the algebraic Sato\--Tate group \cite{banaszak2011algebraic} and the twisting algebraic Sato\--Tate group \cite{fite2014sato}. In \autoref{sec:TwistingST}, we compute the twisting algebraic Sato\--Tate group of the curve $y^2 = x^8 - 14x^4 + 1$.
In \autoref{section:the-right-twist},
we construct our explicit twist of $C_0$ with all automorphisms defined
over $\mathbb Q$ and use this twist to compute explicit models for all other twists
along with the corresponding Sato\--Tate groups.
Then, in \autoref{section:proving-sato-tate},
we prove the generalized Sato\--Tate conjecture for all twists
of $C_0$. In \autoref{sect:MomentSequences}
we compute the moment sequences of all twists of $C_0$.
In Appendix~\ref{section:appendix-determining-sato-tate-group} we include an application of our classification of
twists of $C_0$, showing
that the only Sato\--Tate group of a 3-dimensional abelian variety whose first moment sequence agrees with that of 
$C'$ is $\SU(2)_3$.
Finally, in Appendix~\ref{section:tables},
we give a table summarizing key properties of the
twists of $C$ and 
three tables displaying the first few theoretical and experimental
values of the three
moment sequences associated to the twists of $C$.

\subsection{Acknowledgements} This project was completed as part of a 2016 Arizona Winter School project under the direction of Andrew V. Sutherland. We thank the organizers of the winter school for cultivating an environment which led to the development of this project. 
We also wish to thank Noam Elkies, Francesc Fit\'e, Elisa Lorenzo Garc\'ia, Ken Ono, Andrew V. Sutherland, and David Zureick\--Brown for helpful
conversations. Finally, we thank Andrew V. Sutherland for providing us with a ``photo album" for the Sato\--Tate distributions of our twisted family.\footnote{This can be found at \url{http://math.mit.edu/~drew/aws/g3AWSProjectSatoTateDistributions.html}.} The computations in this paper were performed using \texttt{Magma} \cite{MR1484478}.

\section{Background}
\label{section:background}
After specifying our conventions for the matrix groups $\operatorname{U}(6)$, $\operatorname{Sp}(6)$ and $\operatorname{USp}(6)$,
we recall the definition of the twisting algebraic Sato\--Tate
group and of the twisting Sato\--Tate group. 
\begin{defn}\label{def:AlgGroups}
Set 
\[
J:=\begin{pmatrix}
0 & 1 \\
-1 & 0 \\
& & 0 & 1 \\
& & -1 & 0 \\
& & & & 0 & 1 \\
& & & & -1 & 0
\end{pmatrix}
\]
and for $K$ a field, define 
\[
\operatorname{Sp}(6)/K :=\{ M \in \GL(6)/K : M^t J M = J \}
\]
and
\[
\operatorname{GSp}(6)/K :=\{ M \in \GL(6)/K : M^t J M = \lambda J \text{ for some } \lambda \in K^\times \}.
\]
\noindent We often omit $K$ when $K$ is clear from context. When $K=\mathbb{R}$ is the field of real numbers, we also define 
\begin{align*}
	\operatorname{U}(6)/\mathbb{R} :=\{M \in \GL(6)/\mathbb{C} : M^H M = \id\},
\end{align*}
where $M^H$ is the Hermitian conjugate of $M$. We then define 
\[
	\operatorname{USp}(6)/\mathbb{R} := \operatorname{U}(6)/\mathbb{R} \cap \operatorname{Sp}(6)/\mathbb{R}.
\]
\end{defn}

Let $A$ be an abelian variety of dimension $\leq 3$ defined over $\mQ$. Fix an embedding $\iota\colon \mQ \ext \mC$ and a symplectic basis for the singular homology group $H_1(A_{\mC}^{\text{top}},\mQ)$ where the superscript refers to the underlying topological space. We can use this basis to equip this space with an action of $\GSp(2g)/\mQ$. For each $\tau \in G_{\mathbb Q}$, define the \cdef{twisting Lefschetz group} by

\begin{equation}\label{eqn:twistlef}
\Lefs_A(\tau) := \brk{\gamma \in \Sp(2g) : \gamma\inv\alpha\gamma = {}^{\tau}{}{\alpha} \text{ for all } \alpha \in \End(A_{\overline{\mQ}}) \otimes \mQ};
\end{equation}
when $A=\Jac C$ for some curve $C$, we shall often write $L_C(\tau)$ instead of $L_{\Jac C}(\tau)$.

\begin{defn}\label{def:STgroup}
We define the \cdef{algebraic Sato\--Tate group} of $A$, notated $\AST(A)$, as the union
$$\AST(A) = \bigcup_{\tau \in G_{\mQ}} \Lefs_A(\tau).$$ 
We also define the \cdef{Sato\--Tate group} of $A$, notated $\ST(A)$, as a maximal compact subgroup of $\AST(A)\otimes_{\mQ}\mC$; see \cite[Theorem~6.1, Theorem~6.10]{banaszak2011algebraic}.
\end{defn}

\begin{remark}\label{rmk:quadratic-twisting}
The Sato\--Tate group is not invariant under arbitrary twisting, however, it is invariant under quadratic twisting which acts through the hyperelliptic involution; see \cite[Remark~2.1]{fite2014sato}.
\end{remark}

In this paper, we shall only be concerned with the case where $A = \Jac C$ where $C$ is a genus 3 hyperelliptic curve over $\mQ$, and hence we shall notate $\AST(C) := \AST(\Jac C)$ and $\ST(C) := \ST(\Jac C)$. Using the identification of $H_1$ of the Jacobian with the tangent space to the Jacobian, we may view $\Aut(C)$ as a subgroup of $\GL(H_1(\Jac C_{\mC}^{\text{top}}),\mQ)$. This identification allows us to define the twisting algebraic Sato\--Tate group.

\begin{defn}\label{def:TwistSTgroup}
The \cdef{twisting algebraic Sato\--Tate group} of $C$ is the algebraic subgroup of $\Sp(2g)/\mQ$ defined by
$$\AST_{\Tw}(C) := \AST(\Jac C) \cdot \Aut(C_{\overline{\mQ}}).$$
\end{defn}

Let $\wt{C}$ be a twist of $C$, meaning that $\wt{C}$ is defined over $\mQ$ and there exists a finite extension $L/\mQ$ with $\wt{C}_L \iso C_{L}$. Let $\phi\colon \wt{C}_L \rra C_L$ be a fixed isomorphism. By considering the induced isomorphism $H_1(\Jac\wt{C}_{\mC}^{\text{top}},\mQ) \rightarrow H_1(\Jac C_{\mC}^{\text{top}},\mQ)$, which we denote as $\widetilde{\phi}$, we see that 
\begin{equation}
\Lefs_{\Jac \wt{C}}(\tau) = \widetilde{\phi}\inv\Lefs_{\Jac C}(\tau)({}^\tau \widetilde{\phi}). 
\end{equation}
If we write $\gamma'\in \Lefs_{\Jac \wt{C}}(\tau)$ as $\widetilde{\phi}\inv\gamma({}^{\tau}\widetilde{\phi})$ with $\gamma \in \Lefs_{\Jac C}(\tau)$, then \cite[Lemma~2.3]{fite2014sato} asserts that the map
\begin{align*}
\Lambda_{\widetilde{\phi}}\colon \AST(\Jac \wt{C}) &\lrra \AST_{\Tw}(C) \\
\gamma' &\longmapsto \gamma({}^{\tau}\widetilde{\phi})\widetilde{\phi}\inv
\end{align*}
is a well-defined monomorphism of groups. This lemma allows us to define the twisting Sato\--Tate group of $C$.
\begin{defn}
	\label{definition:twisting-sato-tate-group}
The \cdef{twisting Sato\--Tate group} $\ST_{\Tw}(C)$ of $C$ is a maximal compact subgroup of $\AST_{\Tw}(C)\otimes \mC$.
\end{defn}

\begin{remark}
From \cite[Lemma~2.3]{fite2014sato}, we have that for any twist $\wt{C}$ of $C$, the Sato\--Tate group $\ST(C)$ is isomorphic to a subgroup of $\ST_{\Tw}(C)$. Further, the component groups of $\ST_{\Tw}(C)$ and $\AST_{\Tw}(C)\otimes \mC$ are isomorphic, and the identity components of $\ST_{\Tw}(C)$ and $\ST(C)$ are equal.
\end{remark}

We conclude this section with a proposition which informs us as to which subgroups can arise as Sato\--Tate groups of twists of a particular curve. First, we recall the following simple but fundamental lemma, describing the action of twisting on the twisting Lefschetz group of a curve.

\begin{lemma}[\protect{\cite[Definition~2.20]{fite2012sato}}]\label{lem:howtotwist}
Let $C/\mQ$ be a curve and let $\xi$ be a continuous 1-cocycle $\xi\colon G_{\mQ} \to \operatorname{Aut}(C_{\overline{\mQ}})$. Let $C^{\xi}$ be the twist of $C$ by $\xi$. For all $\tau \in G_{\mQ}$ the following equality holds:
\[
\Lefs_{C^\xi}(\tau) = \Lefs_C(\tau) \xi(\tau)^{-1}.
\]
\end{lemma}

\begin{prop}\label{prop:grouparising}
%Let $G$ be any subgroup of $S_4 \hookrightarrow S_4 \times \ZZ{2}$, which we identify with its image through the composition
%\[
%S_4 \times \ZZ{2} \cong \operatorname{Aut}(C'_{\overline{\mQ}}) \hookrightarrow \operatorname{ST}_{\Tw}(C') \to \operatorname{ST}_{\Tw}(C') / \operatorname{ST}_{\Tw}(C')^0.
%\]
Let $C/\mQ$ be a curve such that $\Aut(C)=\Aut(C_{\overline{\mQ}})$, let $H \subset \Aut(C)$ be a subgroup, and let $G$ denote the image of $H$
under the composition
\[
H \to \operatorname{ST}_{\Tw}(C) \twoheadrightarrow
\operatorname{ST}_{\Tw}(C)/\operatorname{ST}_{\Tw}(C)^0.
\]
Suppose that the inverse Galois
problem for $H$ has a solution. Then there exists a twist $C^G$ of $C$, defined over $\mQ$, such that $\operatorname{ST}(C^G) / \operatorname{ST}(C^G)^0 = G$.
\end{prop}

\begin{remark}
Notice that $\operatorname{ST}( C^G) / \operatorname{ST}( C^G)^0 = G$ is an actual equality and not just an abstract isomorphism. By \cite[Lemma~2.3]{fite2014sato}, we have a canonical monomorphism
\[
\operatorname{ST}(C^G) \hookrightarrow \operatorname{AST}(C^G) \hookrightarrow \operatorname{AST}_{\Tw}(C)
\]
which induces an injection $\operatorname{ST}(C^G)/\operatorname{ST}( C^G)^0 \hookrightarrow \operatorname{ST}_{\Tw}(C^G)/\operatorname{ST}_{\Tw}(C^G)^0$. The equality above is to be interpreted as relative to this embedding.
\end{remark}

\begin{proof}
Recall that twists of $C$ are in bijection with cohomology classes in $H^1(G_{\mQ}, \operatorname{Aut}(C) )$. Given that the Galois action on $\operatorname{Aut}(C)$ is trivial, a cocycle representing such a cohomology class amounts to a continuous homomorphism $\cocyc\colon G_{\mQ} \to \operatorname{Aut}(C)$. 
Fix a Galois extension $L$ of $\mQ$ such that $\operatorname{Gal}(L/\mQ) \cong G$ and choose as $\cocyc$ the canonical projection $G_{\mQ} \to \operatorname{Gal}(L/\mQ)$ followed by an (arbitrary) isomorphism $\operatorname{Gal}(L/\mQ) \cong G$. We shall show that one can take $C^G$ to be $C^\cocyc$, the twist of $C$ corresponding to $\cocyc$.

By our construction of $C$, we know that for all $\tau \in G_{\mQ}$ we have $$\Lefs_{C}(\tau)= \Lefs_{C}(\id)=\Lefs_{C}(\id)^0,$$and hence  $$\operatorname{AST}(C^\cocyc) = \bigcup_{\tau \in G} \Lefs_{C^\cocyc}(\tau) =  \bigcup_{\tau \in G} \Lefs_{C}(\id) \cocyc(\tau)^{-1}.$$ Therefore, $\operatorname{AST}(C^\cocyc)$, seen as a subgroup of $\operatorname{AST}_{Tw}(C^\xi)=\operatorname{AST}_{Tw}( C)$, is precisely the union of those connected components that intersect the image of $\cocyc$. By construction of $\cocyc$, this implies that the group of components of $\operatorname{AST}(C^\cocyc)$ is $G$. Since $\operatorname{ST}(C^\cocyc)$ is a maximal compact subgroup of $\operatorname{AST}(C^\cocyc) \otimes_{\mQ} \mathbb{C}$, we have that
\[
\operatorname{ST}( C^\cocyc) / \operatorname{ST}(C^\cocyc)^0 = \operatorname{AST}(C^\cocyc) / \operatorname{AST}( C^\cocyc)^0 = G.
\]
\end{proof}

\section{The twisting algebraic Sato\--Tate group}\label{sec:TwistingST}
The main result of this section is \autoref{proposition:twisting-sato-tate-group-of-C}, in which
 we determine the twisting Sato\--Tate group for the curve $C$.
In order to compute the twisting Sato\--Tate group, we first compute
the algebraic Sato\--Tate group in~\autoref{lemma:algebraic-sato-tate-group-computation}. To do this, we now compute the decomposition of the Jacobian of $C$
in~\autoref{lemma:jacobian-factorization}.

\begin{lemma}
	\label{lemma:jacobian-factorization}
	The Jacobian of $C$ decomposes as 
$\Jac C \sim (E_1)^2 \times E_2$ where 
\begin{align*}
E_1 \colon y^2 &= x^3 - x^2 - 4x + 4, \\
E_2 \colon y^2 &= x^3 + x^2 - 4x - 4
\end{align*}
are non-CM elliptic curves; moreover $\End(\Jac C) \otimes_{\mZ} \mQ =  \Mat_2(\mQ) \times \mQ$.
\end{lemma}
\begin{proof}
%Since all of the automorphisms are of $C$ are defined over $K:= \mQ(i)$, we have that $\Jac (C)_{\overline{\mQ}} =\Jac(C)_{K} \eqr E^3$ where $E/\mQ$ is some non-CM elliptic curve over $\mQ$. Moreover, we have that $\End(\Jac (C)_{K}) \otimes \mQ = \Mat_3(\mQ)$ since 
%$$\End(\Jac(C)_{K}) \otimes \mQ \cong \operatorname{End}(E^3) \otimes \mQ=\Mat_3(\operatorname{End}(E) \otimes \mQ) = \Mat_3(\mQ).$$
%Since the action of $G_{\mQ} := \Gal(\overline{\mQ}/\mQ)$ on $\End(J_{\overline{\mQ}})_{\mQ}$ factors through $\Gal(\mQ(i)/\mQ)$, with fixed ring $\End(J_{\mQ})_{\mQ}$. 
By computing curve quotients using \texttt{Magma}'s intrinsic 
\begin{verbatim}
CurveQuotient(AutomorphismGroup(C,[a]));
\end{verbatim}
where \texttt{a}$\,\in \Aut(C)$, we see that $\Jac C_{\mQ} \sim E_1 \times E_2 \times E_3$ where $E_1$, $E_2$ are defined above and $E_3$ is some elliptic curve. Since the non-commutative $\Aut(C)$ embeds into $ \End(\Jac C)$, we have that $\End(\Jac C)$ is non-commutative as well, and therefore $E_3$ is isogenous to either $E_1$ or $E_2$. By considering local zeta functions, we conclude that $E_3 \sim E_1$. The statement that $\End(\Jac C) \otimes_{\mZ} \mQ =  \Mat_2(\mQ) \times \mQ$ follows from the curves being non-CM.
\end{proof}

\begin{remark}\label{rmk:endoauto}
	Recall that as mentioned in~\autoref{subsection:an-isolated-point},
all automorphisms of $C$ are defined over $\mathbb Q(i)$.
Analogously to \cite[Lemma~4.2]{fite2014sato}, one can show that the minimal number field over which all the automorphisms of $C_{\overline{\mQ}}$ are defined coincides with the minimal number field over which all the endomorphism of $(\Jac C)_{\overline{\mQ}}$ are defined; since $\Jac C$ decomposes up to isogeny as the product of three non-CM elliptic curves, we do not need to invoke any auxiliary results. 
\end{remark}

\begin{lemma}
	\label{lemma:algebraic-sato-tate-group-computation}
	We have that
$$\AST(C) = \brk{
\pwr{
\begin{array}{c|c|c}
M & {} & {} \\
\hline
{} & M & {} \\
\hline
{} & {} & \pm M
\end{array}
} : M \in \SU(2)}.$$
\end{lemma}
\begin{proof}

Since all of the automorphisms of $C$ are defined over $\bq(i)$, Remark \ref{rmk:endoauto} asserts that all of the endomorphisms of $\Jac C$ are defined over $\mQ(i)$. Moreover, the action of $G_{\mQ}$ on $\End((\Jac C)_{\overline{\mQ}})$ must factor through $\mQ(i)$. To explicitly compute the twisting Lefschetz group, we need to consider how $\Gal(\mQ(i)/\mQ) = \langle \delta \rangle,$ where $\delta$ denotes complex conjugation,
acts on $\End((\Jac C)_{\overline{\mQ}})$.

Since the action of Galois will fix the diagonal matrices, one can check (using
implicitly our description of $\End(\Jac C)$ from~\autoref{lemma:jacobian-factorization})
that Galois acts via
$$\pwr{
\begin{array}{cc|c}
a_{11} & a_{12} & a_{13} \\
a_{21} & a_{22} & a_{23} \\
\hline
a_{31} & a_{32} & a_{33}
\end{array}
} \mapsto  \pwr{
\begin{array}{cc|c}
a_{11} & a_{12} & -a_{13} \\
a_{21} & a_{22} & -a_{23} \\
\hline
-a_{31} & -a_{32} & a_{33}
\end{array}
} .$$

To complete the proof, it suffices to check
\begin{align*}
	\Lefs_{C}(\id) &= 
	\left\{ \pwr{\begin{array}{c|c|c}
	M & 0 & 0 \\ \hline
0 &M & 0 \\ \hline
0 & 0 & M
\end{array}}
: M \in \SL(2)
\right\} ,\\
	\Lefs_{C}(\delta) &= 
	\left\{ \pwr{\begin{array}{c|c|c}
	M & 0 & 0 \\ \hline
0 &M & 0 \\ \hline
0 & 0 & -M
\end{array}}
: M \in \SL(2)
\right\} .
\end{align*}

%Therefore, the algebraic Sato Tate group is the union$L(J(C), \id) \cup L(J(C), \iota)$, where $\iota$ denotes complex conjugation.

We can check these directly from the definition.
To compute $\Lefs(\Jac C, \id)$, taking
\begin{align*}
	\alpha = \pwr{\begin{array}{c|c|c}
	\id & 0 & 0 \\ \hline
0 & 0 & 0 \\ \hline
0 & 0 & 0
	\end{array}},
\gamma = \pwr{
\begin{array}{c|c|c}
a_{11} & a_{12} & a_{13} \\ \hline 
a_{21} & a_{22} & a_{23} \\ \hline
a_{31} & a_{32} & a_{33}
\end{array}}
\end{align*}
the condition that $\gamma \alpha = \alpha \gamma$ implies
$a_{12} = a_{13} = a_{21} = a_{31} = 0$.
Similarly, taking $\alpha$ to be other block matrices with
eight of the blocks equal to $0$, and the ninth equal to the identity,
we obtain $a_{ij} = 0$ if $i \neq j$ and $a_{11} = a_{22} = a_{33}$,
as claimed.
The computation for $\Lefs(\Jac C, \delta)$ is analogous,
this time using that $\gamma^{-1}\alpha\gamma = {}^{\delta}{}{\alpha}$ as $\alpha$ ranges over the matrices with a single nonzero
block which is equal to the identity.
\end{proof}
 
Before we compute the twisting Sato\--Tate group of $C$, we introduce the following notation. Let $$\SL(2)_3:=\left\langle \pwr{
\begin{array}{c|c|c}
M & {} & {} \\
\hline
{} & M & {} \\
\hline
{} & {} & M
\end{array}
} : M \in \SL(2) \right\rangle$$ denote the diagonal embedding of $\SL(2)$ inside $\Sp(6)$; the image of $\SU(2)$ under this embedding will likewise be denoted by $\SU(2)_3$.

\begin{prop}
	\label{proposition:twisting-sato-tate-group-of-C}
The twisting algebraic Sato\--Tate group of $C$ is 
$$\AST_{\Tw}(C) = \left\langle  
\SL(2)_3, \pwr{
\begin{array}{c|c|c}
0 & 0 & \id \\
\hline
\id & 0 & 0 \\
\hline
0 &\id & 0
\end{array}
}, \pwr{
\begin{array}{c|c|c}
0 & \id & 0 \\
\hline
-\id & 0 & 0 \\
\hline
0 & 0 & \id
\end{array}
}  \right\rangle \subset \Sp(6),$$
and hence the twisting Sato\--Tate group of $C$ is 
$$\ST_{\Tw}(C) = \left\langle  \SU(2)_3, 
\pwr{
\begin{array}{c|c|c}
0 & 0 & \id \\
\hline
\id & 0 & 0 \\
\hline
0 & \id & 0
\end{array}
}, \pwr{
\begin{array}{c|c|c}
0 & \id & 0 \\
\hline
-\id & 0 & 0 \\
\hline
0 & 0 & \id
\end{array}
} \right\rangle \subset \USp(6).$$
\end{prop}
\begin{proof}
By Definition \ref{def:TwistSTgroup}, it suffices to realize the generators of $\Aut(C_{\overline \bq})$ as a subgroup of $\Sp(6)$, which we accomplish by computing the action of $\Aut(C_{\overline \bq})$ on $\Jac C$. First, we identify the tangent space to $\Jac C$ at the identity with the space of regular differentials, which has basis given by the following differentials:
\begin{align}
	\label{equation:differentials}
	\omega_1 &:= \frac{2x}{y}dx, &\omega_2&:= \frac{1-x^2}{y}dx, &\omega_3&:= \frac{i(1+x^2)}{y}dx.
\end{align}
Using \texttt{Magma}, we compute that $\Aut(C_{\overline{\mQ}}) = \Aut(C_{\mQ(i)}) = \langle \iota , \alpha_1 , \alpha_2 \rangle$, where $\iota$ is the hyperelliptic involution and $\alpha_1, \alpha_2$ are the automorphisms of order 3 and 4 respectively given by
%\begin{align*}
%\alpha_1\colon (x,y) &\longmapsto \pwr{\frac{-i(x-1)}{(x+1)}, \frac{-4y}{(x+1)^4}}, \\
%\alpha_2\colon (x,y)&\longmapsto \pwr{\frac{-i(x+i)}{(x-i)},\frac{4y}{(x-i)^4}}.
%\end{align*}
\begin{align*}
\alpha_1\colon (x,y) &\longmapsto \pwr{\frac{i(x+1)}{x-1}, \frac{-4y}{(x-1)^4}}, \\
\alpha_2\colon (x,y)&\longmapsto \pwr{\frac{x-1}{x+1},\frac{4y}{(x+1)^4}}.
\end{align*}

By pulling back $\omega_i$ along $\iota$, $\alpha_1$, $\alpha_2$, we realize $\Aut(C_{\overline{\mQ}})$ as the subgroup of $\Sp(6)$ generated by

$$
\left\langle 
\pwr{
\begin{array}{c|c|c}
-\id & 0 & 0 \\
\hline
0 & -\id & 0 \\
\hline
0 & 0 & -\id
\end{array}
},
\pwr{
\begin{array}{c|c|c}
0 & 0 & \id \\
\hline
\id & 0 & 0 \\
\hline
0 & \id & 0
\end{array}
}, \pwr{
\begin{array}{c|c|c}
0 & \id & 0 \\
\hline
-\id & 0 & 0 \\
\hline
0 & 0 & \id
\end{array}
} \right\rangle .
$$
This implies the first statement in the proposition, because $-\id$ is already an element of $\SL(2)_3$.
The second statement follows immediately from the first upon noticing that $\SU(2)$ is a maximal compact subgroup of $\SL(2)$ and it contains $-\id$.
\end{proof}

\begin{remark}\label{remark:identification-of-automorphisms-with-jacobian}
For the rest of the paper, we shall consistently use the differentials $\omega_1, \omega_2, \omega_3$ as defined in
~\eqref{equation:differentials}
as a basis for the tangent space to $\Jac C$ at the identity. This choice of coordinates allows us to identify $\ST_{\Tw}(C)$ with the concrete subgroup of $\USp(6)$ given in the statement of~\autoref{proposition:twisting-sato-tate-group-of-C}.
The explicit description of $\ST_{\Tw}(C)$ also enables to compute its group of components, which is easily seen to be isomorphic to the subgroup of $\operatorname{GL}_3(\mathbb{Z})$ generated by 
\begin{align*}
\sigma:=\pwr{
\begin{array}{c|c|c}
0 & 0 & \id \\
\hline
\id & 0 & 0 \\
\hline
0 & \id & 0
\end{array}
} \hspace*{1em} \tau:=\pwr{
\begin{array}{c|c|c}
0 & \id & 0 \\
\hline
-\id & 0 & 0 \\
\hline
0 & 0 & \id
\end{array}
}. 
\end{align*}
Moreover, we see that $\ST_{\Tw}(C)/\ST_{\Tw}(C)^0$ is isomorphic to $S_4$
by sending $\sigma \mapsto (143)$ and $\tau \mapsto (1234)$. We will find it useful to have one further description of the group of components of $\ST(C)$. Notice that $\sigma$ and $\tau$ generate precisely the (octahedral) subgroup $\operatorname{O}$ of $\GL_3(\mathbb{Z})$ consisting of signed permutation matrices with positive determinant. Furthermore, the group of all $3 \times 3$ signed permutation matrices is $\operatorname{O} \times \mathbb{Z}/2\mathbb{Z}$ (where the factor $\mathbb{Z}/2\mathbb{Z}$ is generated by $-\id$), hence it is isomorphic to $S_4 \times \ZZ{2} \cong \operatorname{Aut}(C_{\overline{\mathbb{Q}}})$. We shall use these identifications to write $\ST_{\Tw}(C)/\ST_{\Tw}(C)^0$ as an index-2 subgroup of $\operatorname{Aut}(C_{\overline{\mathbb{Q}}})$, which in turn we can consider as a subgroup of $\GL_3(\mathbb{Z})$.
\end{remark}

\section{The right twist}
\label{section:the-right-twist}
Given $C$ we now find the twist $C'$ with full geometric automorphism group defined over $\mQ$ and whose Jacobian has connected Sato\--Tate group. 
As described in~\autoref{remark:easy-to-find-c'}, it is not difficult to
show $C'$ is a twist of $C$ with the desired properties, but we prefer to illustrate how one would look for 
$C'$ in the proof of~\autoref{prop:findingtwist}.

\begin{prop}\label{prop:findingtwist}
The curve $C'$ in weighted projective space $\mP_{\mQ}(1,1,1,2)$ defined by
\[
\begin{cases}
x^2+y^2+z^2=0 \\
-2t^2=x^4+y^4+z^4
\end{cases} 
\]
where the variables $x,y,z$ have weight $1$ and $t$ has weight 2 is a twist of $C$ with $\Aut(C') \iso \Aut(C_{\overline{\mQ}})$ and for which $\ST(C)$ is connected. 
\end{prop}
\begin{remark}
	\label{remark:easy-to-find-c'}
It is easy to check using \texttt{Magma}'s intrinsic \texttt{AutomorphismGroup} that the automorphism group of the twist $C'$ is $S_4 \times \ZZ{2}$, but we write out this proof to illustrate how to find the twist with a prescribed connected component of its Sato\--Tate group. In the proof, we show that $C'$ corresponds to the cohomology class in $H^1(G_{\mQ}, \operatorname{Aut}(C_{\overline{\mQ}}))$ of the unique cocycle that factors through $\operatorname{Gal}\left( \mQ(i)/\mQ \right)$ and sends the nontrivial element of this group to the automorphism
\[
(x,y) \mapsto \left(-\frac{1}{x}, - \frac{y}{x^4} \right)
\]
of $C$.
\end{remark}
\begin{proof}
Recall from Lemmas \ref{lemma:algebraic-sato-tate-group-computation} and \ref{lem:howtotwist} that for all $\sigma \in G_{\mQ}$, we have $\Lefs_C(\sigma) = \Lefs_C(\id) \cdot \pi(\sigma)$, where

\begin{align*}
\pi\colon G_{\mQ} &\lrra \GSp(6)/\mQ \\
\sigma &\longmapsto \left\lbrace
\begin{array}{cl}
\id & \mif \sigma \text{ acts trivially on }\mQ(i) ,\\
\pwr{\begin{array}{c|c|c}
\id & {} & {} \\ \hline
{} & \id & {} \\ \hline
{} & {} &-\id
\end{array}} & \text{ otherwise}.
\end{array} \right.
\end{align*}

Let $\cocyc \colon G_{\mQ} \to \operatorname{Aut}(C_{\overline{\mQ}})$ be a cocycle. We have
\begin{equation}\label{eq_TLOfTwist}
\begin{aligned}
\operatorname{AST}(\Jac C^\cocyc) & = \bigcup_{\sigma \in G_{\mQ}} \Lefs_{C^\cocyc} (\sigma) 
 = \bigcup_{\sigma \in G_{\mQ}} \Lefs_{C} (\sigma) \cocyc(\sigma)^{-1} 
 = \bigcup_{\sigma \in G_{\mQ}} \Lefs_{C} (\id) \pi(\sigma) \cocyc(\sigma)^{-1}.
\end{aligned}
\end{equation}
So, in order for $\operatorname{ST}(\Jac C^\cocyc)$ to be connected, it is necessary and sufficient that $\pi(\sigma) \cocyc(\sigma)^{-1}$ be an element of $\Lefs_C(\id)$ for all $\sigma \in G_{\mQ}$.

Now notice that $$\pwr{\begin{array}{c|c|c} \id & {} & {} \\ \hline {} & \id & {} \\ \hline {} & {}& - \id \end{array}} \in \operatorname{Aut}(\operatorname{Jac}C)$$ is induced by the automorphism $(x,y) \mapsto \left(-\frac{1}{x}, -\frac{y}{x^4} \right)$ of $C$, so $\pi$ itself can be interpreted as a cocycle with values in $\operatorname{Aut}(C_{\overline{\mQ}})$. Indeed, it is the unique continuous homomorphism $G_{\mQ} \to \operatorname{Aut}(C_{\overline{\mQ}})$ that factors through $\operatorname{Gal}\left(\mQ(i) / \mQ \right)$ and sends the nontrivial element of this latter group to the automorphism $\psi \colon (x,y) \mapsto \left(-\frac{1}{x}, -\frac{y}{x^4} \right)$. To check that this is indeed a cocycle, we simply notice that $\psi$ is defined over $\mQ$, so a cocycle with values in $\{\id, \psi\}$ is a (continuous) homomorphism with values in this group.

If we now take $\cocyc=\pi$, formula \eqref{eq_TLOfTwist} shows that 
\[
\operatorname{AST}(\Jac C^\cocyc) = \bigcup_{\sigma \in G_{\mQ}} \Lefs_{C} (\id) \pi(\sigma) \pi(\sigma)^{-1} = \Lefs_{C} (\id),
\]
so that $\operatorname{AST}(\Jac C^\cocyc)$ is connected. 
%Since the genus of $C$ is 3, by [FKRS Theorem 2.16] we have $\operatorname{AST}(C^g)=\operatorname{TL}(C^g)$; finally, $\operatorname{ST}(C^g)$ is a maximal compact subgroup of $\operatorname{AST}(C^g) \otimes_{\mQ} \mathbb{C}$, and in particular 
By Definition \ref{def:STgroup}, $\operatorname{ST}(C^\cocyc)$ and $\operatorname{AST}(C^\cocyc)$ have the same group of connected components, and hence the Sato-Tate group of $C^\cocyc$ is connected.

We now want to compute explicit equations for $C^\cocyc$. In order to do this, recall that the function field of $C^\cocyc$ can be identified with the set of fixed points for a certain action of $G_{\mQ}$ on $\overline{\mQ}(C)$. More precisely, an element  $\sigma \in G_{\mQ}$ acts on $\overline{\mQ}(C)=\overline{\mQ}(x,y)$ by the rule
\[
\begin{cases}
\sigma \cdot x=-\frac{1}{x}  \\

\sigma \cdot y=-y/x^4 \\

\sigma \cdot \alpha = \sigma(\alpha) \quad \forall \alpha \in \overline{\mQ}.
\end{cases}
\]

Observe that the three functions 
\begin{align*}
a &:=x-\frac{1}{x}, & b &:=i\pwr{ x  + \frac{1}{x} }, & c &:= i \frac{y}{x^2} 
\end{align*}
are invariant under this Galois action, and we claim that they generate the function field of $\mQ(C^\cocyc)$. To show this, it suffices to prove that (after extending the scalars to $\overline{\mQ}$) they generate $\overline{\mQ}(C)=\overline{\mQ}(x,y)$. This holds because we have 
\begin{align*}
x &=\frac{ia+b}{2}, & y&= \frac{cx^2}{i} = \frac{c}{i} \left( \frac{ia+b}{2} \right)^2.
\end{align*}
Furthermore, the functions $a,b,c$ satisfy $a^2+b^2=-4$ and $c^2 + a^4 + 4a^2 + 16=0$. By considering these equations in weighted projective space $\mP_{\mQ}(1,1,1,2)$, our twist is defined by the equations
\[
\begin{cases}
X^2+Y^2+Z^2=0 \\
T^2-X^2Y^2+Z^4=0,
\end{cases} 
\]
where the variables $X,Y,$ and $Z$ have weight $1$ and $T$ has weight 2. 
Since 
\begin{align*}
X^2Y^2 = \frac{1}{2} (X^2+Y^2)^2 - \frac{1}{2}(X^4+Y^4) = \frac{1}{2}(Z^4-X^4-Y^4)
\end{align*}
we have 
\begin{align*}
	T^2 = X^2Y^2-Z^4 = \frac{1}{2} (Z^4-X^4-Y^4) - Z^4 = -\frac{1}{2}(X^4+Y^4+Z^4).
\end{align*}
So, we can realize the model for our curve $C'$ as
\[
\begin{cases}
X^2+Y^2+Z^2=0 \\
-2T^2=X^4+Y^4+X^4.
\end{cases} 
\]
Replacing $(X,Y,Z,T)$ by $(x,y,z,t)$ yields the result.
\begin{comment}
\[
\begin{cases}
a^2+b^2+c^2=0 \\
d^2-a^2b^2+c^4=0,
\end{cases} 
\]
where the variables $a,b,$ and $c$ have weight $1$ and $d$ has weight 2. 
Since 
\begin{align*}
a^2b^2 = \frac{1}{2} (a^2+b^2)^2 - \frac{1}{2}(a^4+b^4) = \frac{1}{2}(c^4-a^4-b^4)
\end{align*}
we have 
\begin{align*}
	d^2 = a^2b^2-c^4 = \frac{1}{2} (c^4-a^4-b^4) - c^4 = -\frac{1}{2}(a^4+b^4+c^4).
\end{align*}
So, we can realize the model for our curve $C'$ as
\[
\begin{cases}
a^2+b^2+c^2=0 \\
-2d^2=a^4+b^4+c^4.
\end{cases} 
\]
Replacing $(a,b,c,d)$ by $(x,y,z,t)$ yields the result.
\end{comment}
\end{proof}

\begin{remark}
This model helps clarify the action of the automorphism group $S_4\times \ZZ{2}$. The hyperelliptic involution sends $t$ to $-t$, while $S_4$ acts on $(x,y,z)$ through the well-known isomorphism
\[
S_4 \cong \{\pm 1\}^3 \rtimes S_3 / \langle (-1,-1,-1, \id) \rangle;
\] 
that is, an element in $S_4$ induces a permutation of $(x,y,z)$, followed by a change of sign of some of these three coordinates. Notice that in weighted projective space $[x:y:z:t]$ and $[-x:-y:-z:t]$ are the same point, because $t$ has weight 2 while the others have weight one.
\end{remark}

\subsection{Completing the Proof of Theorem~\ref{thm:main}(1)}\label{sec:CompGrp}
In this subsection, we obtain an explicit list of twists of $C'$ whose Sato\--Tate groups realize all possible subgroups of $\ST_{\Tw}(C')$ and prove Theorem~\ref{thm:main}(1).
%In order to prove Theorem~\ref{thm:main}(1), showing that two pairs of twists have isomorphic Sato\--Tate groups.

\begin{lemma}
	\label{lemma:sato-tate-conjugacy-classes}
Conjugacy classes 2 and 3 from Table \ref{Table4} correspond to equivalent Sato-Tate groups, that is, the corresponding Sato-Tate groups are conjugate in $\USp(6)$. More concretely, consider the subgroups of $\operatorname{USp}(6)$ given by
\[
G_2 := \left\langle \SU(2)_3, \pwr{
\begin{array}{c|c|c}
0 & \id & 0 \\
\hline
\id & 0 & 0 \\
\hline
0 & 0 & \id
\end{array}
} \right\rangle \quad \text{ and } \quad G_3 := \left\langle \SU(2)_3, \pwr{
\begin{array}{c|c|c}
-\id & 0 & 0 \\
\hline
0 & \id & 0 \\
\hline
0 & 0 & \id
\end{array}
} \right\rangle;
\]
there exists an element $g \in \operatorname{USp}(6)$ such that $gG_2g^{-1} = G_3$.
Likewise, conjugacy classes 5 and 6 from Table \ref{Table4} also correspond to equivalent Sato-Tate groups, that is, the groups
\[
G_5 := \left\langle \SU(2)_3, \pwr{
\begin{array}{c|c|c}
\pm \id & 0 & 0 \\
\hline
0 & \pm \id & 0 \\
\hline
0 & 0 & \pm \id
\end{array}
} \right\rangle
\]
and
\[
G_6 := \left\langle \SU(2)_3, \pwr{
\begin{array}{c|c|c}
0 & \id & 0 \\
\hline
\id & 0 & 0 \\
\hline
0 & 0 & \id
\end{array}
}, \pwr{
\begin{array}{c|c|c}
-\id & 0 & 0 \\
\hline
0 & -\id & 0 \\
\hline
0 & 0 & \id
\end{array}
} \right\rangle  \]
are conjugate in $\operatorname{USp}(6)$. All other conjugacy classes give pairwise non-isomorphic Sato-Tate groups.
\end{lemma}
\begin{proof}
First, notice that the last statement is immediate, since 
Sato\--Tate groups that have non-isomorphic component groups
correspond to different conjugacy classes. Let
\[
g:= \begin{pmatrix}
0 & 1/\sqrt{2} & 0 & -1/\sqrt{2} & 0 & 0 \\
- 1/\sqrt{2} & 0 & 1/\sqrt{2} & 0 & 0 & 0 \\
0 & - 1/\sqrt{2} & 0 & -1/\sqrt{2} & 0 & 0 \\
 1/\sqrt{2} & 0 & 1/\sqrt{2} & 0 & 0 & 0 \\
0 & 0 & 0 & 0 & 0 & 1 \\
0 & 0 & 0 & 0 & -1 & 0
\end{pmatrix}.
\]
One checks that the equalities $gg^T=\id$ and $g^T J g = J$ hold for $J$ in Definition \ref{def:AlgGroups}, so that $g \in \operatorname{USp}(6)$. Furthermore, we have $$g \pwr{ 
\begin{array}{c|c|c}
0 & \id & 0 \\
\hline
\id & 0 & 0 \\
\hline
0 & 0 & \id
\end{array}
} g^{-1} =\pwr{
\begin{array}{c|c|c}
-\id & 0 & 0 \\
\hline
0 & \id & 0 \\
\hline
0 & 0 & \id
\end{array}
}$$ and
\[
g \operatorname{diag}_3 \begin{psmallmatrix}
a & b \\ c & d
\end{psmallmatrix} g^{-1} = \operatorname{diag}_3 \begin{psmallmatrix}
d & -c \\ -b & a
\end{psmallmatrix},
\]
where by $\operatorname{diag}_3(M)$ denotes the $6 \times 6$ matrix that is block-diagonal with 3 identical $2 \times 2$ blocks given by $M$. This shows that conjugation by $g$ stabilizes $\SU(2)_3$ and sends $G_2$ to $G_3$ as claimed. Finally, one checks that $$g \pwr{
\begin{array}{c|c|c}
-\id & 0 & 0 \\
\hline
0 & -\id & 0 \\
\hline
0 & 0 & \id
\end{array}
} g^{-1}= \pwr{
\begin{array}{c|c|c}
-\id & 0 & 0 \\
\hline
0 & -\id & 0 \\
\hline
0 & 0 & \id
\end{array}
},$$ which (combined with the first part of the lemma) easily implies the claim about $G_5$ and $G_6$.
\end{proof}

We now combine Lemma~\ref{lemma:sato-tate-conjugacy-classes} and
Proposition~\ref{prop:grouparising} to prove Theorem~\ref{thm:main}(1).

\begin{proof}[Proof of Theorem~\ref{thm:main}(1)]
	First, observe that the Sato\--Tate group of a twist is completely characterized by its group of components, given that the identity component is invariant under twisting. Furthermore, the group of components of any twist is a subgroup of $\ST_{\Tw}(C)/\ST_{\Tw}^0(C) = S_4$.
By Proposition~\ref{prop:grouparising}, all the subgroups of $S_4$ can be realized as the group of components of $\ST_{\Tw}(C^{\xi})$, where $C^{\xi}$ is some twist of $C$.
Here, we are using that there is a solution to the inverse Galois problem for every subgroup of $S_4$ (as follows, for example from Shafarevich's result that, that every finite solvable group is realizable as a Galois group of a number field \cite{shafarevich1958imbedding}).
Since conjugation in $S_4$ yields equivalent Sato-Tate groups, there are at most 11 non-equivalent Sato\--Tate groups of twists.
Hence, there are at most 11 Sato\--Tate groups arising from the 11 conjugacy classes of $S_4$, which we recall in Table \ref{Table4}.
By Lemma~\ref{lemma:sato-tate-conjugacy-classes}, conjugation in $\USp(6)$ reduces this number to $9$.
Finally, these $9$ conjugacy classes correspond to pairwise nonisomorphic
Sato\--Tate groups, because the normalized trace distributions
of these twists are pairwise distinct, as computed in Table~\ref{Table4}.
\end{proof}

As in Proposition \ref{prop:findingtwist}, for any subgroup $G$ of $S_4$ we can compute a twist of $C'$ whose Sato\--Tate group has group of components isomorphic to $G$. We now provide another detailed example to illustrate this technique.

\subsection{A $\mathbb{Z}/4\mathbb{Z}$-twist}
We show how to compute a twist of $C'$ whose Sato-Tate group has group of components isomorphic to $\mathbb{Z}/4\mathbb{Z}$. According to section \ref{sec:CompGrp}, a twist $C^\cocyc$ will have this property precisely if the image of $\cocyc \colon \abGal \to \Aut(C')$ is a cyclic group of order $4$, which means that the kernel of $\cocyc$ defines a degree-4 cyclic extension of $\mQ$. Among all such extensions, the one with smallest discriminant is the cyclotomic field $\mQ(\zeta)$, where $\zeta$ is a primitive $5\tth$ root of unity. Since we are free to choose $\cocyc$, we take $\ker \cocyc$ to be $\operatorname{Gal}\left( \overline{\mQ} / \mQ(\zeta) \right)$. To completely describe $\cocyc$, we now need to choose an injective homomorphism of $\operatorname{Gal}\left(\mQ(\zeta)/\mQ \right)$ into $\Aut(C')$. As a generator of $\operatorname{Gal}\left(\mQ(\zeta)/\mQ \right)$, we take the unique automorphism $\sigma$ of $\mQ(\zeta)/\mQ$ that sends $\zeta$ to $\zeta^2$, and we define our $\cocyc$ by requiring that $\varphi := \cocyc(\sigma) \in \Aut(C')$ acts on the function field $\mQ(C')$ as follows:
\[
\varphi(t/z^2)=t/z^2, \quad \varphi(x/z)=-y/z, \quad \varphi(y/z)=x/z.
\]

In order to compute equations for the twist $C^\cocyc$, simply recall (see for example \cite[Chapter X.2]{Silverman}) that the function field $\mQ(C^\cocyc)$ is the fixed field of $\mQ(\zeta)(C')$ under the twisted Galois action given by
\[
\begin{array}{lcl}
\sigma(\zeta)=\zeta^2 & & \sigma(x/z)= \varphi(x/z) = -y/z \\
\sigma(y/z)=\varphi(y/z)=x/z & & \sigma(t/z^2)=\varphi(t/z^2)=t/z^2.
\end{array}
\]

We now determine the field of invariants for this action. Clearly $t/z^2$ is fixed under Galois. For all $j \in \mathbb{Z}$ we have another invariant function given by
\[
\sum_{g \in \operatorname{Gal}\left( \mQ(\zeta)/\mQ \right)} g \cdot (\zeta^j x/z).
\]
Taking $j=1$ and $j=2$, we find that the functions $u$ and $v$ given by
\begin{equation}\label{eq_xyab}
\begin{pmatrix} u \\ v \end{pmatrix} := \left(
\begin{array}{cc}
 \zeta -\zeta ^4 & \zeta^3-\zeta ^2 \\
 \zeta ^2 (1-\zeta ) & \zeta  \left(1-\zeta ^3\right) \\
\end{array}
\right) \begin{pmatrix} x/z \\ y/z \end{pmatrix}
\end{equation}
are Galois invariants.
Since the matrix appearing in this formula is invertible over $\mQ(\zeta)$, the function fields $\mQ(\zeta)(u,v,t/z^2)$ and $\mQ(\zeta)(x/z,y/z,t/z^2)$ coincide. Hence the curve (defined over $\mQ$) whose function field is $\mQ(u,v,t/z^2)$ becomes isomorphic to $C'$ over $\mQ(\zeta)$; in other words, we have already found generators for the full field of invariants. Now we just need to find the equations satisfied by $u,v$ and $t/z^2$. Defining $M$ to be the $2 \times 2$ matrix appearing in \eqref{eq_xyab}, we can write
\[
\begin{pmatrix} x/z \\ y/z \end{pmatrix} = M^{-1} \begin{pmatrix}u  \\ v \end{pmatrix}=\frac{1}{1 - \zeta^2 - 2 \zeta^3 + 2 \zeta^4}\left(
\begin{array}{c}
 \left(\zeta ^2+\zeta +1\right) u+\zeta  v \\
\zeta  u-\left(\zeta ^2+\zeta +1\right) v  \\
\end{array}
\right),
\]
and we know that $x/z, y/z$ and $t/z^2$ satisfy the two equations
\[
\begin{cases}
(x/z)^2+(y/z)^2+1=0 \\
-2(t/z^2)^2 = (x/z)^4 + (y/z)^4 + 1.
\end{cases}
\]

Replacing $x/z, y/z$ by their expressions in terms of $u, v$ and expanding, the previous system becomes
\[
\begin{cases}
u^2+v^2-5 = 0 \\
(t/z^2)^2 = \displaystyle \frac{1}{125} \left(3 u^4+4 u^3 v+12 u^2 v^2-4 u v^3+3 v^4\right)+1 .
\end{cases}
\]
In particular, the first equation defines a conic section with a $\mQ$-rational point, so this curve admits a hyperelliptic model over $\mQ$; \texttt{Magma}'s intrinsic \texttt{IsHyperelliptic} returns the hyperelliptic model
\begin{equation}\label{eq:ModelZ4}
y^2 = -5x^8 - 5x^7 - 35x^6 + 35x^5 - 35x^3 -
    35x^2 + 5x - 5.
    \end{equation}
Finally, since quadratic twists do not change the Sato\--Tate group of a curve (cf.~Remark \ref{rmk:quadratic-twisting}), and since $\sqrt{5}$ belongs to $\mQ(\zeta)$, we can further take the quadratic twist by 5 of \eqref{eq:ModelZ4} without changing either the splitting field of the twist or the group of components of the corresponding Sato\--Tate group. Thus, we arrive at the following model for the $\ZZ{4}$-twist:
\[
y^2 = -x^8 - x^7 - 7x^6 + 7x^5 - 7x^3 -
    7x^2 + x - 1.\]
    
\section{Proving Sato\--Tate for all the twists}
\label{section:proving-sato-tate}
The goal of this section is to prove the generalized Sato\--Tate conjecture for all twists of $C$,
which will complete the proof of our main theorem~\autoref{thm:main},
by proving~\autoref{thm:main}(3).
We accomplish this at the end of the section in~\autoref{thm:ProofST}. To prove this, our method is to observe the equality of the distribution of the $a_i(A)(p)$ and the pushforward of the Haar measure $\mu$ by $\Phi_i$ through direct computation. The computation of the distribution of $a_i(A)(p)$ as $A$ ranges over the twists of $C$ is done in~\autoref{prop:MomentSequences3}, while the computation
of the pushforward of the Haar measure is given in~\autoref{prop:MomentSequences4}.

Before proceeding, we briefly recall the definitions of equidistribution and moment sequences. Let $X$ be a compact topological space equipped with a Radon measure $\mu$. Denote by $C(X)$ the normed space of continuous, complex valued functions $f$ on $X$ with norm given by $\|f\| := \sup_{x \in X}|f(x)|$. We say that a sequence $\{x_i\}_{i \in \mathbb{N}}$ is \cdef{equidistributed} with respect to $\mu$ if for all $f \in C(X)$ we have

\[\mu(f) = \lim_{m \rightarrow \infty}\dfrac{1}{m}\sum_{i=1}^m f(x_i). \]

We restrict our attention to the case where $X$ is an interval. In this case, the \cdef{$n\tth$ moment}, if it exists, is defined as
\[M_n(\{x_i\}) = \lim_{m \rightarrow \infty}\dfrac{1}{m}\sum_{i \le m} x_i^n.\]
Notice, that if the sequence $\{x_i\}$ is equidistributed with respect to the measure $\mu$, we have that the $n\tth$ moment exists and is equal to $\mu(\psi_n)$ where $\psi_n(x) = x^n$. In fact a partial converse statement holds. If the $n\tth$ moment of a sequence $\{x_i\}$ exists for every $n \ge 0$, then there exists a unique measure $\mu$ such that $\{x_i\}$ is equidistributed with respect to $\mu$ (see Remark 2.5 of \cite{sutherland2016sato} for details).

We notate the moments associated to the measure $\mu$ by $\mE_{\mu}[x^n]$. For later use, we denote by $n_*\mu$ the pushforward of $\mu$ by the multiplication-by-$n$ map; in particular, if $\mu$ is supported on the interval $[-a,a]$, then $n_*\mu$ is supported on $[-na,na]$. Further define $\phi$ to be the Sato\--Tate distribution of a non-CM elliptic curve (which is given by $\phi(f) = \frac{1}{2\pi}\int_{-2}^2 f(x) \sqrt{4-x^2} d x$ for $f \in C([-2,2])$) and define $\delta_0$ to be the point mass at $0$, which we will use in the tables in \autoref{section:tables}.

\begin{remark}\label{rmk:MomentsPhi}
A straightforward direct computation gives
\[
\mathbb{E}_\phi[x^n] = \Gamma\left(\frac{n+1}{2} \right) \Gamma\left(\frac{n+4}{2} \right)^{-1} (2\sqrt{\pi})^{-1} \cdot (2^n+(-2)^n),
\]
and by definition we have $\mathbb{E}_{3_*\phi}[x^n]=3^n\mathbb{E}_\phi[x^n]$.
\end{remark}

\subsection{Finding the distribution of traces of Frobenius}
\label{subsection:distribution-of-trace}

Before computing the distribution of $a_i(A)(p)$ for all $i$ in~\autoref{subsection:higher-i-distribution},
we first work
out the instructive special case when $i = 1$ in the main result of this subsection, \autoref{prop:MomentSequences2}. That is, we are looking for the distribution
of the normalized trace of Frobenius.

Let $C^\xi$ be a twist of $C'$ corresponding to a twisting cocycle $\xi\colon\abGal \to \operatorname{Aut}(C')$.
Let $E/\mQ$ be the elliptic curve $y^2=x^4-14x^2+1$. We know that $\operatorname{Jac} C^\xi$ is $\overline{\mQ}$-isogenous to $E^3$. Moreover, for every prime $p$ we know by \cite[Proposition 11]{MR2678623} that the action of the Frobenius at $p$ on $\operatorname{Jac}C^\xi$ is given by $F_p \circ \xi(\operatorname{Frob}_p)$, where $F_p$ is the automorphism of $\operatorname{Jac} C$ induced by the Frobenius at $p$ (notice that the Frobenius at $p$ acts on $(\operatorname{Jac}C)_{\overline{\mathbb{F}_p}} \cong (\operatorname{Jac}C^\xi)_{\overline{\mathbb{F}_p}}$).
Considering the matrix representation of the automorphisms of the form $\xi(\operatorname{Frob}_p)$, we can understand the trace of Frobenius acting on $C^{\xi}$ in terms of the trace of Frobenius for $E$.
\begin{lemma}\label{lemma:functions-fg}
Let $\ell \neq p$ be some prime. For each $g \in \operatorname{Aut}(C')$, there is a function $f_g\colon\mathbb{R} \to \mathbb{R}$ such that 
\[
\frac{1}{\sqrt{p}} \operatorname{tr} \left(  \operatorname{Frob}_p \bigm\vert H^1_{\acute{e}t}( (\Jac C^\xi)_{\overline{\mathbb{F}_p}},\mQ_\ell) \right) = f_g\left( \frac{1}{\sqrt{p}}  \operatorname{tr} \left( \operatorname{Frob}_p \bigm\vert H^1_{\acute{e}t}(E_{\overline{\mathbb{F}_p}},\mQ_\ell) \right) \right)
\]
whenever $\xi(\operatorname{Frob}_p)=g$. Moreover, $f_g$ depends only on the conjugacy class of $g$ in $\operatorname{Aut}(C')$; more precisely, we have
$$
f_g(x)= \left\lbrace \begin{array}{ll}
3x & \text{ if } g=\id ,\\
-3x & \text{ if } g=\iota ,\\
\pm x &\text{ if $g \neq \id$ has order 2 or 4},\\
0 & \text{ if $g$ has order 3 or 6}.
\end{array}\right.
$$
\end{lemma}
\begin{proof}
Notice that the matrix representation of $F_p$ has the form $A \otimes \id_3$, where $A$ is the $2 \times 2$ matrix representation of $\operatorname{Frob}_p$ acting on $H^1_{\acute{e}t}(E_{\overline{\mathbb{F}_p}},\mQ_\ell)$. Also notice that $\xi(\operatorname{Frob}_p) \in \operatorname{GL}_6(\mathbb{Z})$ is of the form $\id_2 \otimes \Pi$, with $\Pi$ a signed permutation matrix: indeed, the action of Frobenius on $\Jac(C^\xi)$ permutes the three 1-dimensional factors in the decomposition $\Jac(C^\xi)_{\overline{\mathbb{Q}}} \sim E^3$ and acts on each of them via an automorphism, which (since $E$ does not have CM) can only be $\pm \id$. %Clearly the matrix $\Pi$ is determined by $\xi(\operatorname{Frob}_p)$
Since the action of the Frobenius at $p$ on $H^1_{\acute{e}t}( (\Jac C^\xi)_{\overline{\mathbb{F}_p}},\mQ_\ell)$ is given by $F_p \circ \xi(\operatorname{Frob}_p)$, we obtain
\[
\begin{aligned}
\frac{1}{\sqrt{p}} \operatorname{tr} \left(  \operatorname{Frob}_p \bigm\vert H^1_{\acute{e}t}( (\Jac C^\xi)_{\overline{\mathbb{F}_p}},\mQ_\ell) \right) & = \frac{1}{\sqrt{p}}\operatorname{tr} \left( (A \otimes \id_3)(\id_2 \otimes \Pi) \right)\\
& = \frac{1}{\sqrt{p}} \operatorname{tr}(A)\operatorname{tr}(\Pi) \\
& =\frac{1}{\sqrt{p}} \operatorname{tr}\left( \operatorname{Frob}_p  \bigm\vert H^1_{\acute{e}t}(E_{\overline{\mathbb{F}_p}},\mQ_\ell) \right) \cdot \operatorname{tr}\Pi,
\end{aligned}
\]
which proves the lemma. Notice that a signed permutation matrix of order 2 or 4 has trace $\pm 1$ (unless it is $-\id$, which has trace $-3$), while a signed permutation matrix of order 3 or 6 has trace 0. For more details, see also \cite[Proposition~12]{MR2678623}.
%This lemma follows immediately from Table \ref{table:phi-coefficients}. More precisely the function $f_g$ can be computed using this table; 
\end{proof}

%Since $f_g$ only depends on $g$ through its conjugacy class in $\xi()$

\subsubsection{Further Notation}
\label{subsubsection:notation-for-distributions}
For simplicity, denote by $b_p$ the normalized trace of the Frobenius at $p$ acting on $E$ and by $a^\xi_p$ the normalized trace of $\operatorname{Frob}_p$ acting on $\operatorname{Jac} C^\xi$. We shall also write  
\[
f(g,p) := f_g(b_p)=f_g\left( \frac{1}{\sqrt{p}} \operatorname{tr} \left( \operatorname{Frob}_p \bigm\vert H^1_{\acute{e}t}(E_{\overline{\mathbb{F}_p}},\mQ_\ell) \right) \right).
\]

These quantities make sense for all but finitely many primes; let $S$ be the set of bad primes at which either $E$ or $C^\xi$ has bad reduction.

We now compute the limiting distribution for the normalized traces of Frobenius $a_p^{\xi}$ for the twist $C^\xi$ (hence for all the twists, since we haven't made any assumptions about $\xi$). Let $K_\xi$ be the fixed field of the kernel of $\xi$, that is the minimal extension of $\mQ$ splitting the twist. By passing to the quotient, $\xi$ induces an isomorphism $\operatorname{Gal}(K_\xi/\mQ) \cong \xi(\abGal) \subseteq \Aut(C')$. 

\begin{lemma}
	\label{lemma:sum-over-conjugacy-classes}
	Keeping the notation above, let $H:=\xi(\abGal)$, and let $\mathcal D$ be the set of conjugacy classes of $H$. For $D \in \mathcal{D}$, set $f(D,p):=f(h,p)$, where $h \in H$ is any element of $D$.
	The definition is well-posed, and for any $k \in \mathbb N$,
	\begin{align*}
\sum_{p \leq x, p\not\in S} (a^\xi_p)^k = \sum_{D \in \mathcal{D}} \sum_{\substack{p \leq x, p \not \in S \\ \xi(\operatorname{Frob}_p) \in D}} f(D,p)^k.
	\end{align*}
\end{lemma}
\begin{proof}
	We trivially have (for any $k \in \mathbb{N}$)
\[
\sum_{p \leq x, p\not\in S} (a^\xi_p)^k = \sum_{p \leq x, p \not \in S} f(\xi(\operatorname{Frob}_p),p)^k = \sum_{h \in H} \sum_{\substack{p \leq x, p \not \in S \\ \xi(\operatorname{Frob}_p)=h}} f(h,p)^k.
\]
Now observe that $f_h$ depends on $h$ only through its conjugacy class in $\operatorname{Aut}(C')$, so a fortiori, it can only depend on $h$ through its conjugacy class in $\xi(\abGal)$. Indeed, as $\xi(\abGal)$ is a subgroup of $\operatorname{Aut}(C')$, conjugacy in $\xi(\abGal)$ is a finer invariant than conjugacy in $\Aut(C')$. If $D$ is a conjugacy class in $H$, then, it makes sense to set $f(D,p):=f(h,p)$ where $h$ is any element of $D$. 

It follows that if we let $\mathcal{D}$ be the set of conjugacy classes of $H$ we have
\begin{align*}
\sum_{p \leq x, p\not\in S} (a^\xi_p)^k & = \sum_{D \in \mathcal{D}} \sum_{h \in D} \sum_{\substack{p \leq x, p \not \in S \\ \xi(\operatorname{Frob}_p)=h}} f(h,p)^k \\
& = \sum_{D \in \mathcal{D}} \sum_{h \in \mathcal{D}} \sum_{\substack{p \leq x, p \not \in S \\ \xi(\operatorname{Frob}_p)=h}} f(D,p)^k
\\& = \sum_{D \in \mathcal{D}} \sum_{\substack{p \leq x, p \not \in S \\ \xi(\operatorname{Frob}_p) \in D}} f(D,p)^k.
\end{align*}
\end{proof}

\begin{lemma}
	\label{lemma:average-trace}
	Retain the notation of~\autoref{subsubsection:notation-for-distributions} and~\autoref{lemma:sum-over-conjugacy-classes}. For each conjugacy class $D$ in $H$, define $f_D(x):=f_h(x)$, where $h$ is any element in $D$. Let $\pi(x)$ denote the number of primes bounded by $x$. We have
\[
\begin{aligned}
\frac{1}{\pi(x)} \sum_{p \leq x, p\not\in S} (a^\xi_p)^k & = \sum_{D \in \mathcal{D}} \frac{|D|}{|H|} \left( \int_{\operatorname{SU}(2)} f_D(\operatorname{tr}(g))^k dg  \right) +o(1).
\end{aligned}
\]

\end{lemma}
\begin{proof}

By the refined Sato\--Tate conjecture for $E$, which asserts the equidistribution of $b_p$ when we restrict to primes $p$ having a given Artin symbol in the extension $K_\xi/\mQ$ (cf.~\cite[Theorem 1]{MR3012726}), we know that
\begin{equation}\label{eq_ST2}
\sum_{\substack{p \leq x, p \not \in S \\ \xi(\operatorname{Frob}_p) \in D}} f(D,p)^k =  \left( \int_{\operatorname{SU}(2)} f_D(\operatorname{tr}(g))^k dg  \right) \left(\sum_{\substack{p \leq x, p \not \in S \\ \xi(\operatorname{Frob}_p) \in D}} 1\right) + o_D \left(\sum_{\substack{p \leq x, p \not \in S \\ \xi(\operatorname{Frob}_p) \in D}} 1\right)
\end{equation}
as $x \to \infty$. We now apply Chebotarev's theorem, in the form
\[
\frac{1}{\pi(x)} \left(\sum_{\substack{p \leq x, p \not \in S \\ \xi(\operatorname{Frob}_p) \in D}} 1\right) = \frac{|D|}{|H|} + o_D(1) \text{ as } x \to \infty.
\]

Dividing equation \eqref{eq_ST2} by $\pi(x)$, we find
\[
\frac{1}{\pi(x)} \sum_{\substack{p \leq x, p \not \in S \\ \xi(\operatorname{Frob}_p) \in D}} f(D,p)^k = \frac{|D|}{|H|} \left( \int_{\operatorname{SU}(2)} f_D(\operatorname{tr}(g))^k dg  \right) +o_D(1).
\]
By summing over $D \in \mathcal{D}$ and using Lemma~\ref{lemma:sum-over-conjugacy-classes}, we deduce
\[
\begin{aligned}
\frac{1}{\pi(x)} \sum_{p \leq x, p\not\in S} (a^\xi_p)^k & = \frac{1}{\pi(x)} \sum_{D \in \mathcal{D}} \sum_{\substack{p \leq x, p \not \in S \\ \xi(\operatorname{Frob}_p) \in D}} f(D,p)^k \\
& = \sum_{D \in \mathcal{D}} \frac{|D|}{|H|} \left( \int_{\operatorname{SU}(2)} f_D(\operatorname{tr}(g))^k dg  \right) +o(1).
\end{aligned}
\]
\end{proof}

In Proposition~\ref{prop:MomentSequences2}, we compute the limiting distribution of $a_p^{\xi}$ over the group of components of a twist $C^{\xi}$ of $C'$, which is $H/(H\cap \iota)$ where $H = \xi(G_{\mQ})$. Notice that in Lemmas~\ref{lemma:sum-over-conjugacy-classes}, \ref{lemma:average-trace}, our computations involved the group $H$ not $H/(H\cap \iota)$. We accommodate for this discrepancy below by considering elements of $H\setminus\brk{\iota}$, not simply $H$.

\begin{prop}\label{prop:MomentSequences2}
With notation as in~\autoref{subsubsection:notation-for-distributions}, let $C^\xi$ be the twist of $C'$ corresponding to a homomorphism $\xi\colon \abGal \to \operatorname{Aut}(C')$, let $H$ be the image of $\xi$, let $b$ be the number of elements in $H \setminus \{\iota\}$ having order $2$ or $4$, and let $c$ be 1 (respectively 0) if $H$ contains (respectively does not contain) the hyperelliptic involution $\iota$. Finally let $M_k := \int_{\operatorname{SU}(2)} \operatorname{tr}(g)^k dg$.
For every $k>0$, we have
\begin{equation}\label{eq_ST3}
\lim_{x \to \infty} \frac{1}{\pi(x)} \sum_{p \leq x, p\not\in S} (a^\xi_p)^k = \frac{3^k (1+c) +b}{|H|}M_k,
\end{equation}
while for $k=0$ we trivially have $\lim_{x \to \infty} \frac{1}{\pi(x)} \sum_{p \leq x, p\not\in S} (a^\xi_p)^0 = 1$.
\end{prop}
\begin{proof}
Consider the elements of $H$ as signed permutation matrices, and recall that under this identification $\iota$ corresponds to $-\id$.
By~\autoref{lemma:functions-fg},
$$
f_D(x)= \left\lbrace \begin{array}{ll}
\pm 3x &\text{ if } D=\{\pm \id\},\\
\pm x & \text{ if the elements of $D \neq \{\pm \id\}$ have order 2 or 4},\\
0 &\text{ if the elements of D have order 3 or 6}.
\end{array} \right.
$$
Using this and Lemma~\ref{lemma:average-trace}, we have
\[
\begin{aligned}
\frac{1}{\pi(x)} \sum_{p \leq x, p\not\in S} (a^\xi_p)^k & =  \sum_{D \in \mathcal{D}} \frac{|D|}{|H|} \left( \int_{\operatorname{SU}(2)} f_D(\operatorname{tr}(g))^k dg  \right) +o(1);
\end{aligned}
\]
if $k$ is odd, the integral over $\SU(2)$ vanishes and we obtain $\frac{1}{\pi(x)} \sum_{p \leq x, p\not\in S} (a^\xi_p)^k =0$, which proves \eqref{eq_ST3} in this case (since $M_k=0$ for all odd $k$). If $k$ is even, then the sign in the definition of $f_D(x)$ is not important, and we obtain
\[
\begin{aligned}
\frac{1}{\pi(x)} \sum_{p \leq x, p\not\in S} (a^\xi_p)^k= \frac{|\{h \in H\setminus\{\iota\} : \operatorname{ord}(h) =2 \text{ or } 4 \}|}{|H|} M_k + \frac{1+c}{|H|} \cdot 3^kM_k + o(1).
\end{aligned}
\]\end{proof}

\subsection{Distributions of all $a_i$}
\label{subsection:higher-i-distribution}

We now generalize~\autoref{prop:MomentSequences2} to all coefficients of the characteristic polynomial of the normalized Frobenius automorphism.
This is accomplished in the
main result of this subsection,~\autoref{prop:MomentSequences3}.

\begin{lemma}\label{lemma:higher-i-distrib}
	Let $F_p^\xi$ be the automorphism of $H^1_{\acute{e}t}( (\operatorname{Jac}C^\xi)_{\overline{\mathbb{F}_p}}, \mQ_\ell )$ given by the action of the Frobenius at $p$, and let $H_p^{\xi}(t)$ be the characteristic polynomial of $F_p^\xi$ and $h_p^\xi(t):= \frac{1}{p^3} H_p^\xi(\sqrt{p}t)$. The polynomial $h_p^\xi(t)$, which we call the \textbf{normalized characteristic polynomial of the Frobenius at $p$}, can be computed purely in terms of $b_p$ and of $\xi(\operatorname{Frob}_p)$.
\end{lemma}
\begin{proof}
The $6 \times 6$ matrix $\operatorname{diag}_3(A)$ representing the automorphism $F_p$ is block-diagonal, each of the three identical $2 \times 2$ blocks $A$ being given by the action of Frobenius on the 2-dimensional vector space $H^1_{\acute{e}t}(E_{\overline{\mathbb{F}_p}},\mQ_\ell)$. As in the proof of Lemma \ref{lemma:functions-fg}, one sees that the matrix $M$ representing the action of $F_p^\xi$ is the tensor product of $A$ with a certain $3 \times 3$ signed permutation matrix $\Pi$ which only depends on $\xi(\operatorname{Frob}_p)$. Let $\alpha_1,\alpha_2$ be the roots of the characteristic polynomial of $A$ and $\pi_1, \pi_2, \pi_3$ be the roots of the characteristic polynomial of $\Pi$. The characteristic polynomial $H_p^\xi(t)$ of $M=A \otimes \Pi$ is then given by
\[
H_p^\xi(t)=\prod_{j=1}^3 \prod_{i=1}^2 (t-\alpha_i \pi_j) = \prod_{j=1}^3 (t^2-(\alpha_1+\alpha_2)\pi_j t+\alpha_1\alpha_2 \pi_j^2).
\]
Thus, we have that
\begin{equation}\label{eq:CharPoly}
\begin{aligned}
h_p^\xi(t) & = \frac{1}{p^3} \prod_{j=1}^3 (p t^2-\sqrt{p}(\alpha_1+\alpha_2)\pi_j t+p \pi_j^2) \\
& = \prod_{j=1}^3 (t^2-\frac{\alpha_1+\alpha_2}{\sqrt{p}}\pi_j t+\pi_j^2) \\
& = \prod_{j=1}^3 (t^2-b_p\pi_j t+\pi_j^2).
\end{aligned}
\end{equation}

Since the $\pi_j$ are determined by $\Pi$, which in turn is determined by $\xi(\operatorname{Frob}_p)$, this proves the lemma.
\end{proof}
As an immediate consequence of the explicit expression for $h_p^\xi(t)$ given by equation \eqref{eq:CharPoly}, we also have the following result, which generalizes Lemma \ref{lemma:functions-fg}:
\begin{coro}\label{cor:higher-distrib-functions}
There exist polynomials $f_{g,i}(x)$ for $i=0, \ldots, 6$ and $g \in \operatorname{Aut}(C')$ such that the following holds: for every cocycle $\xi\colon \abGal \to \Aut(C')$ and every prime $p$, the normalized characteristic polynomial of the Frobenius at $p$ acting on $H^1_{\acute{e}t}((\Jac C^\xi)_{\overline{\mathbb{F}_p}},\mQ_\ell)$ is given by
\[
h^\xi_p(t)=\sum_{i=0}^6 f_{\xi(\operatorname{Frob}_p),i}(b_p) t^i.
\]
Moreover, the polynomials $f_{g,i}(x)$ only depend on $g$ through its conjugacy class in $\Aut(C') \iso S_4 \times \ZZ{2} \subset \operatorname{GL}_3(\mathbb{Z})$, and are given explicitly as follows:
\begin{itemize}
\item $f_{g,0}(x)=f_{g,6}(x)= 1$;
\item $f_{g,1}(x)=f_{g,5}(x)=-\operatorname{tr}(g) x$;
\item $\displaystyle f_{g,2}(x)=f_{g,4}(x) = \operatorname{tr}(g^2) + \frac{x^2}{2} \left( \operatorname{tr}(g)^2 - \operatorname{tr}(g^2) \right)$;
\item $f_{g,3}(x)=-\frac{1}{3} \left( \operatorname{tr}(g)^3-\operatorname{tr}(g^3)-6 \det(g) \right) x - x^3 \det(g)$.
\end{itemize}
Here in order to compute the trace and determinant of $g$ we consider it as a $3 \times 3$ signed permutation matrix.
\end{coro}

\begin{remark}\label{rem:FunctionsF}
More generally, formula \eqref{eq:CharPoly} shows the following: if $g \in \GL_3(\mathbb{Z})$ is a signed permutation matrix and $A$ is an arbitrary element of $\SU(2)$, then $f_{g,i}(\operatorname{tr} A)$ is the $i\tth$ coefficient of the characteristic polynomial of $A \otimes g$.
\end{remark}

Since the functions $f_{g,i}$ are clearly continuous, the same argument that leads to Proposition \ref{prop:MomentSequences2} shows
\begin{prop}\label{prop:MomentSequences3}
Retaining the notation from Proposition \ref{prop:MomentSequences2}, let 
\[
h^\xi_p(t)=a_{6,p}^\xi t^6+a_{5,p}^\xi t^5+\cdots+a_{1,p}^\xi t+a_{0,p}^\xi
\]
denote the normalized characteristic polynomial of the Frobenius at $p$ acting on $\operatorname{Jac}C^\xi$. For each conjugacy class $D$ in $H$, define $f_{D,i}(x):=f_{h,i}(x)$, where $h$ is any element of $D$ of positive determinant. The functions $f_{D,i}(x)$ are well-defined, and  for each $i=0,\ldots,6$ and every $k \geq 0$ we have
\[
\lim_{x \to \infty} \frac{1}{\pi(x)} \sum_{p \leq x, p \not \in S} (a_{i,p}^\xi)^k = \sum_{D \in \mathcal{D}} \frac{|D|}{|H|} \int_{\SU(2)} f_{D,i}(\operatorname{tr}(x))^k dx.
\]
\end{prop}

\subsection{Completing the Proof}
\label{subsection:completing-the-proof}
For each of the Sato\--Tate groups corresponding to twists of our curve $C'$, we  find in~\autoref{prop:MomentSequences4} the pushforward of the Haar measure along each coefficient of the normalized characteristic polynomial. We conclude our proof of the main Theorem~\ref{thm:main} in~\autoref{thm:ProofST}.

\begin{prop}\label{prop:MomentSequences4}
Let $G$ be a subgroup of $\operatorname{ST}_{\operatorname{Tw}}(C)$ such that $G^0=\operatorname{ST}_{\operatorname{Tw}}(C)^0$; let $H$ be the group of components of $G$, seen as a subgroup of $S_4 \subset \{x \in \GL_3(\mathbb{Z}) : \det x =1 \}$; let $\Phi_i \colon G \to \mathbb{R}$ for $i=0,\ldots, 6$ be the map sending $g \in G$ to the $i\tth$ coefficient of the characteristic polynomial of $g$; let $\mu$ be the Haar measure on $G$. By pushing forward $\mu$ along $\Phi_i$, we obtain probability measures $\mu_i$ on $\mathbb{R}$, and we have
\[
\mathbb{E}_{\mu_i}[x^k] = \sum_{D \in \mathcal{D}} \frac{|D|}{|H|} \int_{\SU(2)} f_{D,i}(\operatorname{tr}(x))^k dx
\]
where $k \geq 0$ and the functions $f_{D,i}$ were defined in Corollary \ref{cor:higher-distrib-functions}.
\end{prop}

\begin{proof}
%Just as in the proof of Lemma \ref{lemma:higher-i-distrib}, one sees that 
For every connected component $M$ of $G$, there is a signed permutation matrix $\Pi_M \in \GL_3(\mathbb{Z})$ such that
\[
A \in \SU(2) \mapsto A \otimes \Pi_M \in M
\]
is a bianalytic bijection between $\SU(2)$ and $M$. Furthermore, one sees that the characteristic polynomial of $A \otimes \Pi_M$ is determined by $A$ and by $\Pi_M$. Thus one obtains
\[
\begin{aligned}
\mathbb{E}_{\mu_i}[x^k]
& = \int_{G} \Phi_i(x)^n \, d\mu_{\ST(C^\xi)}(x) \\
& = \sum_{M} \int_M  \Phi_i(x)^n \, d\mu_{\ST(C^\xi)}(x) \\
& = \sum_{M} \frac{1}{[G:G^0]} \int_{\SU(2)}  \Phi_i(A \otimes \Pi_M)^n \, d\mu_{SU(2)}(A) \\
& = \sum_{M} \frac{1}{|H|} \int_{\SU(2)}  f_{M,i}(\operatorname{tr}(A))^n \, d\mu_{SU(2)}(A) ,
\end{aligned}
\]
where we have set $f_{M,i}(x):=f_{\Pi_M,i}(x)$ and $f_{g,i}(x)$ is defined in the statement of Corollary \ref{cor:higher-distrib-functions} (notice that we can consider $\Pi_M$ as an element of $S_4 \times \mathbb{Z}/2\mathbb{Z} \subset \GL_3(\mathbb{Z})$, see Remark \ref{remark:identification-of-automorphisms-with-jacobian}). We obtain the formula in the statement by grouping together elements that belong to the same conjugacy class in $H$.
\end{proof}

\begin{theorem}\label{thm:ProofST}
Let $C^\xi$ be any twist of $C'$. The generalized Sato\--Tate conjecture is true for $C^\xi$.
\end{theorem}
\begin{proof}
Let $a_{i,p}^\xi$ be the $i\tth$ coefficient of the normalized characteristic polynomial of the Frobenius at $p$ acting on $\Jac C^\xi$, $\mu$ be the Haar measure of $\operatorname{ST}(C^\xi)$, and $\Phi_i \colon \operatorname{ST}(C^\xi) \to \mathbb{R}$ be the map giving the $i\tth$ coefficient of the characteristic polynomial of an endomorphism. Propositions \ref{prop:MomentSequences3} and \ref{prop:MomentSequences4} imply that for $i=0,\ldots, 6$ the moments of the sequence $(a_{i,p}^\xi)_{p }$ are the same as the moments of the measure $\Phi_{i*} (\mu)$.
This implies that for fixed $\xi$ and $i$ the sequence $(a_{i,p}^\xi)_{p }$ is equidistributed 
with respect to $\Phi_{i*} (\mu)$.% Since almost all Frobenius elements act semisimply, and therefore are determined up to conjugacy by their characteristic polynomials, this also proves equidistribution of the normalized Frobenius elements in the space of conjugacy classes of $\operatorname{ST}(\Jac C^\xi)$ equipped with the measure induced by $\mu$, that is, the Sato-Tate conjecture for $C^\xi$.
\end{proof}

\section{Computing moment sequences of twists}\label{sect:MomentSequences}
Below in Proposition \ref{prop:computemoments}, we compute the Sato\--Tate distribution of the each twist by summing the contribution corresponding to the cycle types over each component in the Sato\--Tate group
of that twist; the complete list of distributions is given in
Table \ref{Table4}. We also obtain explicit formulas for the moment sequences for the coefficients $a_i(A)(p)$ for $i = 1,2,3$ of each twist.

\begin{prop}\label{prop:computemoments}
Let $C^\xi$ be a twist of $C'$ and let $G$ be the component group of $\operatorname{ST}(C^\xi)$; for $j=2,3,4$, let $m_j$  be the number of elements in $G$ of order $j$, $\mu$ be the Haar measure of $\ST(C^\xi)$. Let $\Phi_i \colon \operatorname{ST}(C^\xi) \to \mathbb{R}$ be the map giving the $i\tth$ coefficient of the characteristic polynomial of an endomorphism; let $\mu_i$ be the pushforward of $\mu$ along $\Phi_i$. For $n>0$, the moment sequences of $\mu_1, \mu_2,$ and $ \mu_3$ are given by:
\begin{align*}
& \mE_{\mu_1}[x^n] = \frac{(3^n+m_2+m_4)}{|G|}\cdot \mE_{\phi}[x^n],\\
&\mE_{\mu_2}[x^n]=\frac{1}{|G|} \mathbb{E}_\phi[(3+3x^2)^n] + \frac{m_2}{|G|} \mathbb{E}_\phi[(3-x^2)^n] + \frac{m_4}{|G|} \mathbb{E}_\phi[(-1+x^2)^n], \\
&\mE_{\mu_3}[x^n]=
\frac{1}{|G|} \mathbb{E}_\phi[(x^3+6x)^n] + \frac{m_2+m_4}{|G|} \mathbb{E}_\phi[(x^3-2x)^n] + \frac{m_3}{|G|} \mathbb{E}_\phi[(3x-x^3)^n]
\end{align*}
for $n=0$ we trivially have $\mE_{\mu_i}[x^0]=1$ for $i=1,2,3$.
\end{prop}

\begin{remark}
Recall that $\ST(C^\xi)$ is a subgroup of $\SU(6)$, so the characteristic polynomial of any $g \in \ST(C^\xi)$ is symmetric, that is, we have $\Phi_i=\Phi_{6-i}$ and therefore $\mu_i=\mu_{6-i}$. In particular, \autoref{prop:computemoments} also gives the moment sequences of the measures $\mu_4, \mu_5$. Finally, $\Phi_0=\Phi_6$ is the constant function 1, so the moments of $\mu_0=\mu_6$ are all equal to 1.
\end{remark}

\begin{proof}
Our purpose is to compute, for every $n \geq 0$, the integrals
\begin{align*}
\int_{\ST(C^\xi)} \Phi_1(x)^n \, d\mu_{\ST(C^\xi)}(x) , \quad
\int_{\ST(C^\xi)} \Phi_2(x)^n \, d\mu_{\ST(C^\xi)}(x)  ,  \quad \int_{\ST(C^\xi)} \Phi_3(x)^n \, d\mu_{\ST(C^\xi)}(x) ,
\end{align*}
which as above we write as $\sum_{M} \int_M \Phi_i(x)^n \, dx$.
Notice first that the functions $\Phi_i \colon \ST(C^\xi) \to \mathbb{R}$ are invariant under conjugation in $\operatorname{GL}_6(\mathbb{R})$, so that in particular they are also invariant under conjugation in $\ST_{\Tw}(C)$ since $\ST(C^\xi) \subseteq \ST_{\Tw}(C) \subset \GL_6(\mathbb{R})$. 

Secondly, observe that the Haar measure of $\ST(C^\xi)$ is invariant under conjugation as well, which implies that different connected components of $\ST(C^\xi)$ that are conjugated in $\ST_{\Tw}(C)$ give the same contribution to the integral $\int_{\ST(C^\xi)} \Phi_i(x)^n \, d\mu_{\ST(C^\xi)}(x)$. Since the connected components of $\ST(C^\xi)$ are a subset of the connected components of $\ST_{\Tw}(C)$, it suffices to compute
\[
\int_M \Phi_i(x)^n \, d\mu_{\ST_{\Tw}(C)}(x)
\]
for every connected component $M$ of $\ST_{\Tw}(C)$. If $M$ is a connected component of $\ST(C^\xi)$, then
\[
\int_M \Phi_i(x)^n \, d\mu_{\ST(C^\xi)}(x) = \frac{|\ST_{\Tw}(C)/\ST_{\Tw}(C)^0|}{|G|} \int_M \Phi_i(x)^n \, d\mu_{\ST_{\Tw}(C)}(x)
\]
because the total mass of $M$ for the Haar measure of $\ST(C^\xi)$ (resp.~$\ST_{\Tw}(C)$) is $1/|G|$ (resp.~$1/|\ST_{\Tw}(C)/\ST_{\Tw}(C)^0|$), so we will need to account for this rescaling in what follows.

%, and then sum the result over those connected components of $\ST_{\Tw}(C)$ that are also connected components of $\ST(C^\xi)$. 

Now notice that if $\sigma_1, \sigma_2 \in S_4$ are permutations that index two connected components $M_1, M_2$ of $\ST_{\Tw}(C)$, then $M_1, M_2$ are conjugate in $\ST_{\Tw}(C)$ if and only if $\sigma_1, \sigma_2$ are conjugate in $S_4$, which occurs if and only if $\sigma_1$, $\sigma_2$ have the same cycle type. Thus, it suffices to understand the contribution of the five families of connected components, corresponding to the five cycle-types in $S_4$.
\begin{enumerate}
\item \underline{Identity permutation}, $M=\ST(C^\xi)^0$: In this case, we have $\Pi_M=\id_3$, the $3 \times 3$ identity matrix. In order to employ the formulas from Corollary \ref{cor:higher-distrib-functions}, we first compute that $\det(\Pi_M)=1$ and $\operatorname{tr}(\Pi_M)=\operatorname{tr}(\Pi_M^2)=\operatorname{tr}(\Pi_M^3)=3$. We obtain $f_{\Pi_M,1}(x)= 3x$, $f_{\Pi_M,2}(x)=3+3x^2$, and $f_{\Pi_M, 3}(x)=-x^3-6x$, that is, the equalities
\begin{align*}
\Phi_1(A\otimes \Pi_M) = f_{\Pi_M,1}(\operatorname{tr}(A)) &= -3\operatorname{tr}(A), \\
\Phi_2(A \otimes \Pi_M)=f_{\Pi_M,2}(\operatorname{tr}(A))&=3+3\operatorname{tr}(A)^2 ,\\
\Phi_3(A \otimes \Pi_M)=f_{\Pi_M,3}(\operatorname{tr}(A))&=-\operatorname{tr}(A)^3-6\operatorname{tr}(A),
\end{align*}
hold for all $A \in \SU(2)$. Now observe that $\ST_{\Tw}(C)$ has 24 connected components and that $\ST_{\Tw}(C)^0$ is isomorphic to $\SU(2)$, and hence the restriction of $\mu_{\ST_{\Tw}(C)}$ to $\ST_{\Tw}(C)^0 $ is $\frac{1}{24}$-th of the standard Haar measure on $\SU(2)$. 
It follows that
\begin{align*}
|G| \int_{\ST(C^\xi)^0} \Phi_1(x)^n \, d\mu_{\ST(C^\xi)}(x)
& = 24 \int_{\ST_{\Tw}(C)^0} \Phi_1(A \otimes \id)^n \, d\mu_{\ST_{\Tw}(C)}(A) \\
& = \int_{\SU(2)} \Phi_1(A \otimes \id)^n \, d\mu_{SU(2)}(A) \\
& =\int_{\SU(2)} (-3\operatorname{tr}(A))^n \, d\mu_{SU(2)}(A) \\
& = \mathbb{E}_\phi\left[(-3x)^n\right] = \mathbb{E}_\phi\left[(3x)^n\right],
\end{align*}
\[
\begin{aligned}
|G| \int_{\ST(C^\xi)^0} \Phi_2(x)^n \, d\mu_{\ST(C^\xi)}(x)
& = \int_{\SU(2)} \Phi_2(A \otimes \id)^n \, d\mu_{SU(2)}(A) \\
& =\int_{\SU(2)} (3+3\operatorname{tr}(A)^2)^n \, d\mu_{SU(2)}(A)  \\ &= \mathbb{E}_\phi\left[3^n(1+x^2)^n\right],
\end{aligned}
\]
and likewise
\[
\begin{aligned}
|G| \int_{\ST(C^\xi)^0} \Phi_3(x)^n \, dx & = \int_{\SU(2)} \Phi_3(A \otimes \id)^n \, d\mu_{SU(2)}(A) \\
& =\int_{\SU(2)} (-\operatorname{tr}(A)^3-6\operatorname{tr}(A))^n \, d\mu_{SU(2)}(A) \\
& = \mathbb{E}_\phi\left[(-1)^n (x^3+6x)^n\right] \\
& = \mathbb{E}_\phi\left[(x^3+6x)^n\right],
\end{aligned}
\]
where in the last equality we have used the fact that $\phi$ is symmetric with respect to zero.
%where the factor $|G|$ in front of the first integral comes from the fact that the Haar measure of $\ST(C^\xi)^0$ is $1/|G|$. 

\item \underline{Order-three permutation}, $M=\begin{psmallmatrix} & & \id \\
\id \\
& \id \end{psmallmatrix} \ST(C^\xi)^0$: In this case we can take $\Pi_M=\begin{psmallmatrix} & & 1 \\ 1 \\ & 1 \end{psmallmatrix}$. We have $\operatorname{tr}(\Pi_M)=\operatorname{tr}(\Pi_M^2)=0$, $\operatorname{tr}(\Pi_M^3)=3$ and $\det(\Pi_M)=1$, whence $f_{\Pi_M,1}(x)=0$, $f_{\Pi_M,2}(x)=0$ and $f_{\Pi_M,3}(x)=3x-x^3$. Proceeding as in the previous case, we obtain
\begin{align*}
\int_M \Phi_1(x)^n \, d\mu_{\ST(C^\xi)}(x) &= 0 ,\\
\int_M \Phi_2(x)^n \, d\mu_{\ST(C^\xi)}(x) &= 0 ,\\
\int_M \Phi_3(x)^n d\mu_{\ST(C^\xi)}(x) &= \frac{1}{|G|}\mathbb{E}_\phi[(3x-x^3)^n].
\end{align*}
\item \underline{Transposition}, $M=\begin{psmallmatrix}
& \id \\ \id \\ & & \id
\end{psmallmatrix}\ST(C^\xi)^0$: We have $\Pi_M=\begin{psmallmatrix}
& 1 \\ 1 \\ & & 1
\end{psmallmatrix}$, $\operatorname{tr}(\Pi_M)=\operatorname{tr}(\Pi_M^3)=1$, $\operatorname{tr}(\Pi_M^2)=3$, and $\det(\Pi_M)=-1$. It follows that $f_{\Pi_M,1}(x)=-x$, $f_{\Pi_M,2}(x)=3-x^2$ and $f_{\Pi_M,3}(x)=x^3-2x$, whence
\begin{align*}
\int_M \Phi_1(x)^n \, d\mu_{\ST(C^\xi)}(x) &= \frac{1}{|G|}\mathbb{E}_{\phi}[x^n], \\
\int_M \Phi_2(x)^n \, d\mu_{\ST(C^\xi)}(x) &= \frac{1}{|G|}\mathbb{E}_{\phi}[(3-x^2)^n] ,\\
\int_M \Phi_3(x)^n d\mu_{\ST(C^\xi)}(x) &= \frac{1}{|G|}\mathbb{E}_\phi[(x^3-2x)^n].
\end{align*}
\item \underline{Double transposition},  $M=\begin{psmallmatrix}
\id \\ & \id \\ & & -\id
\end{psmallmatrix}\ST(C^\xi)^0$: We have $\Pi_M=\begin{psmallmatrix}
1 \\ & 1 \\ & & -1
\end{psmallmatrix}$, and since $\begin{psmallmatrix}
& 1 \\ 1 \\ & & 1
\end{psmallmatrix}$ and $\Pi_M$ are similar in $\operatorname{GL}_3(\mathbb{R})$, we have the same values for each of the integrals as in (3).
\item \underline{Four-cycle}, $M=\begin{psmallmatrix}
& -\id \\ \id \\ & & \id
\end{psmallmatrix}\ST(C^\xi)^0$: By setting $\Pi_M=\begin{psmallmatrix}
& -1 \\ 1 \\ & & 1
\end{psmallmatrix}$, we can compute that $\operatorname{tr}(\Pi_M)=\operatorname{tr}(\Pi_M^3)=1$, $\operatorname{tr}(\Pi_M^2)=-1$, $\det(\Pi_M)=1$, and whence $f_{\Pi_M,1}(x)=-x$, $f_{\Pi_M,2}(x)=x^2-1$ and $f_{\Pi_M,3}(x)=2x-x^3$. Notice that
\begin{align*}
\int_{\SU(2)} (2\operatorname{tr}(A)-\operatorname{tr}(A)^3)^n \, d\mu_{SU(2)}(A) = \int_{\SU(2)} (-2\operatorname{tr}(A)+\operatorname{tr}(A)^3)^n \, d\mu_{SU(2)}(A)
\end{align*}
since both integrals vanish if $n$ is odd, and they are clearly equal if $n$ is even. Thus we obtain
\begin{align*}
\int_M \Phi_1(x)^n \, d\mu_{\ST(C^\xi)}(x) &= \frac{1}{|G|}\mathbb{E}_{\phi}[x^n], \\
\int_M \Phi_2(x)^n \, d\mu_{\ST(C^\xi)}(x) &= \frac{1}{|G|}\mathbb{E}_{\phi}[(x^2-1)^n] ,\\
\int_M \Phi_3(x)^n d\mu_{\ST(C^\xi)}(x) &= \frac{1}{|G|}\mathbb{E}_\phi[(x^3-2x)^n].
\end{align*}
\end{enumerate}
Finally, we need to sum over the connected components of each type: clearly there is precisely one component of type (1), and there are $m_3$ components of type (2), $m_2$ components of types (3) and (4), and $m_4$ components of type (5). From the above computations, our result follows.
%shows that the coefficients $\Phi_2, \Phi_3$ of the normalized characteristic polynomial of $\operatorname{Frob}_p$ are deterministic functions of $b_p$ and of the automorphism induced by $\operatorname{Frob}_p$ on $\Jac(C^\xi)$.
\end{proof}

\begin{remark}
Notice that since $-\id$ belongs to $\ST(C^\xi)^0$, for every connected component $M$ of $\ST(C^\xi)$ there are two possible choices of $\Pi_M$, namely the ones we have used and their opposites. However, different choices lead (as they should) to the same result: this is clear for the case of $\Phi_2$, which, being a quadratic function of the eigenvalues, takes on the same value on a matrix $x$ and on its opposite $-x$. It is also true for $\Phi_1$ and $\Phi_3$, because all the odd-indexed moments vanish, while all the expressions of the form $\Phi_{1}^{2k}, \Phi_{3}^{2k}$ are again even functions of their argument, hence insensitive to our arbitrary choice of sign.
\end{remark}

\begin{remark}
	\label{remark:}
	In Tables~\ref{table:mu-1},~\ref{table:mu-2}, and~\ref{table:mu-3}, we present theoretical and numerical computations of the first few terms of the moment sequences of $\mu_1$, $\mu_2$, and $\mu_3$. The algorithm used to compute the numerical data comes from \cite{harvey2016computing}.

Notice that the curves corresponding to subgroups $4$, $5$, and $6$ in Table \ref{Table4} all have the same $\mu_1$ moments, but not the same Sato\--Tate distribution. Indeed, the $\mu_2$ moments provide a distinction between the curve corresponding to subgroup 4 and the curves corresponding to subgroups $5$, and $6$, where the latter two curves have the same Sato\--Tate distribution by Lemma \ref{lemma:sato-tate-conjugacy-classes}.
\end{remark}

\clearpage
\appendix

\section{Determination of Sato\--Tate group from the first Moment Sequence}
\label{section:appendix-determining-sato-tate-group}
In general, the sequence of Frobenius traces does not determine the Sato\--Tate group. However, in the case where the Sato\--Tate distribution is $3_*\phi$, as defined at the beginning of Section~\ref{section:proving-sato-tate}, we show that this distribution can only arise from the Sato\--Tate group $\SU(2)_3$. In fact we prove something stronger: if the trace distribution of the Sato-Tate group of a 3-dimensional abelian variety $A$ has the same first five moments as $3_*\phi$, then $\ST(A)$ is conjugate to $\SU(2)_3$ in $\USp(6)$.

\begin{prop}
	\label{proposition:appendix-main}
Let $A/K$ be a 3-dimensional abelian variety over some number field $K$. Suppose that the first five moments of the trace distribution $\mu_1$ associated with $\ST(A)$ coincide with the first five moments of $3_*\phi$, that is, we have $\mE_{\mu_1}[x^n] = \mE_{3_*\phi}[x^n]$ for $n=0,\ldots,4$. Then $A$ is $K$-isogenous to the cube of an elliptic curve without CM, and $\ST(A)$ is conjugate to $\SU(2)_3$ inside $\GSp(6)$. 
\end{prop}

The proof of this result occupies the remainder of the section.
Notate $G:=\ST(A)$. Here, $G$ is a (reductive) compact real Lie group, acting naturally on $\mathbb{R}^6 \subseteq \mathbb{C}^6$. Let $W$ be $\mathbb{C}^6$ interpreted as a representation of $G$, and write $W=\bigoplus_{i=1}^k W_i^{\oplus n_i}$ for the decomposition of $W$ into irreducible $G$-modules. Let $\chi=\sum_{i=1}^k n_i \chi_i$ be the character of this representation $W$. Observe that $\chi$ is real, and hence the assumption on the moments sequence implies 
\[
\langle \chi, \chi \rangle = \int_G \operatorname{tr}(g)^2 dg = \mathbb{E}_{\mu_1}[x^2]= 9 \Rightarrow \sum_{i=1}^k n_i^2 = 9,
\]
using $\mathbb{E}_{3_*\phi}[x^2]=9$ (see Remark \ref{rmk:MomentsPhi}).
On the other hand, we have that $\sum n_i \dim W_i = 6$, which leaves precisely two possibilities:
\begin{enumerate}
\item either $k=1$, $n_1=3$ and $\dim W_1=2$,
\item or $k=3$, $n_1=n_2=2$, $n_3=1$, $\dim W_1=\dim W_2=1$ and $\dim W_3=2$.
\end{enumerate}
Let $\mathfrak{g}$ be the Lie algebra of $G_\mathbb{C}$, and write $\mathfrak{g}=\bigoplus_{j=1}^m \mathfrak{g}_j \oplus \mathfrak{c}$ with every $\mathfrak{g}_i$ simple and $\mathfrak{c}$ abelian. Each representation $W_i$ can be written as $W_{i1} \boxtimes \cdots \boxtimes W_{im} \boxtimes \psi_i$, where $W_{ij}$ is an irreducible representation of $\mathfrak{g}_j$ and $\psi_i$ is a character of $\mathfrak{c}$. Notice that $W$ is a faithful representation of $G$, so for each $j=1,\ldots,m$ there exists an $i$ such that $W_i$ is a faithful representation of $\mathfrak{g}_j$. In particular, $W_{ij}$ is a faithful representation of $\mathfrak{g}_j$, and since $\dim W_{ij} \leq 2$ we obtain that $\mathfrak{g}_j$ has a faithful representation of dimension at most 2. This is only possible if $\mathfrak{g}_j$ is $\mathfrak{sl}_2$ for all $j$, because no other simple Lie algebra admits an irreducible 2-dimensional faithful representation.

In both cases $(1)$ and $(2)$, up to isomorphism (of $G$-modules) there is only one irreducible submodule $W_i$ of dimension 2, which implies that $m$ (the number of copies of $\mathfrak{sl}_2$ appearing in $\mathfrak{g}$) is at most 1. Also observe that $m$ cannot be zero, because an abelian Lie algebra does not admit any irreducible 2-dimensional faithful representation over $\mathbb{C}$. We can therefore write $\mathfrak{g}=\mathfrak{sl}_2 \oplus \mathfrak{c}$ and $W_i = W_{i1} \boxtimes \psi_i$, where $W_{i1}$ is either trivial or isomorphic to the standard representation of $\mathfrak{sl}_2$ and $\psi_i$ is a character of $\mathfrak{c}$.
To complete the proof, we will now verify
Proposition~\ref{proposition:appendix-main} holds
in both cases $(1)$ and $(2)$.

\begin{enumerate}
\item Suppose that we are in case (1) above. 
We claim that $\mathfrak{c}$ is trivial. To show this, we restrict our attention to a maximal compact Lie subalgebra of $\mathfrak{g}$, which is necessarily of the form $\mathfrak{su}_2 \oplus \mathfrak{c}_\mathbb{R}$. By faithfulness, there exists $i$ such that $\psi_i$ is nontrivial; as $W$ comes from a real representation and $W_{i1} \boxtimes \psi_i \subseteq W$ we also have $\overline{W_{i1}} \boxtimes \overline{\psi_i} \subseteq W$. We have $\overline{W_{i1}}=W_{i1}$ both for the trivial and the standard representation of $\mathfrak{su}_2$, so $W_{i1} \boxtimes \overline{\psi_i} \subseteq W$. On the other hand we have $\overline{\psi}_i \neq \psi_i$: indeed, $\psi_i$ is obtained (by differentiation and extension of scalars) from a real representation $\rho_i$ of the maximal central torus $\mathbb{T}$ of $\ST(A)$. As $\ST(A)$ is compact, $\mathbb{T}$ is a product of copies of $\mathbb{S}^1$, hence its representation $\rho_i$ satisfies $\overline{\rho_i}=\rho_i^{-1}$, because all the eigenvalues of $\rho_i(x)$ lie on the unit circle, for all $x \in \mathbb{T}$. Extending scalars and differentiating we find $\overline{\psi_i}=-\psi_i$, so in particular $\overline{\psi_i} \neq \psi_i$ since $\psi_i$ is nontrivial. This immediately leads to a contradiction, because on the one hand we have $W_{i1} \boxtimes \overline{\psi_i} \not \cong W_{i1} \boxtimes \psi_i$ and on the other all irreducible submodules of $W$ are isomorphic. Hence $\mathfrak{c}=(0)$. Since $\mathfrak{c}=(0)$, we obtain that $G^0$ is abstractly isomorphic to $\SU(2)$, which in turn implies that the Hodge group of $A$ is (abstractly isomorphic to) $\SL(2)$. 

To complete this case, we will show that in fact $\ST(A)$
is not only isomorphic to $\SU(2)$, but in fact
is conjugate to $\SU(2)_3$ inside $\GSp(6)$.
Write $A_{\overline{k}}$ (up to isogeny) as a product of simple abelian varieties, $A_{\overline{K}} \sim \prod_{i=1}^s A_i^{t_i}$. If $s=1$, this leads immediately to a contradiction, because no absolutely simple abelian threefold has Hodge group isomorphic to $\operatorname{SL}(2)$. If instead $A$ is nonsimple, all its absolutely simple factors are of dimension at most 2. By \cite[Corollary~1.2]{lombardo2014ell}, we have
\[
\SL(2) \cong \operatorname{Hg}(A) \cong \prod_{i=1}^s \operatorname{Hg}(A_i),
\]
where each factor on the right is nontrivial (because no abelian variety has trivial Hodge group). This clearly implies $s=1$, $\operatorname{Hg}(A_1)=\SL(2)$, and $t_1 \dim A_1=3$: it is immediate to check that this is only possible if $A_1$ is an elliptic curve without CM and $t_1=3$. In turn, this means that $A$ is \textit{geometrically} isogenous to the cube of a non-CM elliptic curve, so that $\ST(A)^0=\SU(2)_3$ (up to conjugacy in $\GSp(6)$). Finally, Table~\ref{table:mu-1} shows that for a Sato\--Tate group with this identity component we can have $\mathbb{E}_{\mu_1}[x^2]=9$ only when $\ST(A)$ is connected. That is, if and only if $\ST(A)=\ST(A)^0=\SU(2)_3$ as claimed. Notice that the argument for this case only involves the first three moments of $\mu_1$, and also shows that the equality $\ST(A)=\SU(2)_3$ is equivalent to $A$ being $K$-isogenous to the cube of an elliptic curve without CM.
\item In case (2), the same argument as above leads to the conclusion that $\overline{W_3}=W_3$, so that $W_1=\overline{W_2}$ is given by a character $\psi$ of $\mathfrak{c}$, while $\mathfrak{c}$ acts trivially on $W_3$ (which otherwise would not be self-conjugate).  If $\mathfrak{c}$ is trivial, the same argument as in the previous case leads to the conclusion that $A$ is $K$-isogenous to the cube of an elliptic curve. This is absurd under the hypotheses of case (2), because $W$ would then be isomorphic to the direct sum of three identical copies of the same representation, which is incompatible with $W_1 \not \cong W_3$. To conclude the proof, we will assume $\mathfrak{c}$ is nontrivial and derive a contradiction with the assumption $\mathbb{E}_{\mu_1}[x^4]=162$. Notice that there exist abelian threefolds for which we have $W_1 = \overline{W_2}$, $\overline{W_3}=W_3$, and $\mathfrak{c}$ nontrivial: an example is given by $A = E_1^2 \times E_2$, where $E_1$ is an elliptic curve admitting CM over $K$ and $E_2$ is an elliptic curve without potential CM. As we now show, however, the Sato-Tate group of any such threefold does not satisfy $\mathbb{E}_{\mu_1}[x^4]=162$.

	We write the character $\chi$ of $W$ as $2\psi + 2\overline{\psi} + \eta$, where $\psi$ is a group morphism $G \to \mathbb{C}^\times$ and $\eta$ is the character of an irreducible, 2-dimensional, real representation of $G$. 
\begin{lemma}\label{lem:int}
The integrals $\int_G \psi^n \eta^k dg$ and $\int_G \overline{\psi}^n \eta^k dg$ vanish for all $n>0$ and $k \geq 0$.
\end{lemma}
\begin{proof}
Notice that the statements for $\psi$ and $\overline{\psi}$ are clearly equivalent, so we only consider the former case. For the sake of simplicity, we identify representations with their characters; in particular, if $\omega$ is the character of a representation $V$, then $\omega^k$ is identified with $V^{\otimes k}$.

Recall that $\eta$ is real, so the number $\int_G \psi^n \eta^k dg = \langle \psi^n, \eta^k \rangle$ is the multiplicity of the 1-dimensional representation $\psi^n$ as a constituent of the representation $\eta^{k}$. We want to prove that this number is zero, so it suffices to show that $\psi^n$ is not isomorphic to a $G^0$-submodule of $\eta^k$ (a fortiori, this implies that there is no $G$-equivariant, nonzero map from $\psi^n$ to $\eta^{k}$). It also suffices to show that $\psi^n$ is not a subrepresentation of $\eta^k$ for the action of the (real) Lie algebra $\mathfrak{g}_{\mathbb{R}}$ of $G^0$. We have already shown that $\mathfrak{g}_{\mathbb{R}}=\mathfrak{su}_2 \oplus \mathfrak{c}_{\mathbb{R}}$, and that $\eta=\operatorname{Std} \boxtimes 1$, where $\operatorname{Std}$ is the standard 2-dimensional representation of $\mathfrak{su}_2$ and $1$ denotes the trivial representation of $\mathfrak{c}$. It follows easily that for $k \geq 0$ the only possible 1-dimensional constituents of $\eta^k=(\operatorname{Std})^{\otimes k} \boxtimes 1$ are copies of the trivial representation of $\mathfrak{g}_\mathbb{R}$, which, in particular, are not isomorphic to the non-trivial 1-dimensional representation $\psi^n$. It follows that $\langle \psi^n, \eta^k \rangle$ vanishes, as we wanted to show.
\end{proof}

We now compute $\int_G \chi^4 dg$ in two different ways. On the one hand, this is the fourth moment of the Sato\--Tate distribution of traces associated with $G$, so it is equal to $\mathbb{E}_{3_*\phi}[x^4]=162$ (see \autoref{rmk:MomentsPhi})
by assumption. On the other hand, using Lemma \ref{lem:int} and the equality $\psi\overline{\psi}=1$, we find
\[
\begin{aligned}
\int_G \chi^4 dg & = \int_G \left(48 \eta ^2 \psi  \overline{\psi }+96 \psi ^2 \overline{\psi }^2 +\eta ^4\right) dg \\
& = \int_G \left(48 \eta ^2 +96 +\eta ^4\right) dg \\
& = 96 + 48\langle \eta, \eta \rangle + \int_G \eta^4 dg \\
& = 144 + \int_G \eta^4 dg.
\end{aligned}
\]
Observe now that $\eta$ is a $2$-dimensional representation, so $|\eta(g)| \leq 2$ for all $g \in G$. It follows that $|\int_G \eta^4 dg| \leq \int_G 2^4 dg = 16$, whence $162=\int_G \chi^4 dg \leq 144+16=160$. This is a contradiction, completing the
proof of Proposition~\ref{proposition:appendix-main}.
\end{enumerate}

\clearpage
\section{Tables}
\label{section:tables}

{\renewcommand{\arraystretch}{1.3}%
\begin{table}[h!]
	\centering
	\begin{tabular}{|c|ccccc|}
		\hline
		Subgroup & $M_2$ & $M_4$ & $M_6$ & $M_8$ & $M_{10}$ \\ \hline \hline
		1 (theoretical) & 9& 162& 3645& 91854& 2480058\\
		1 (numerical) & 8.999& 161.988& 3644.890& 91854.573& 2480122.102 \\ \hline		
		2 (theoretical) & 5& 82& 1825& 45934& 1240050\\
		2 (numerical) &4.999 & 81.973& 1824.235& 45910.323& 1239239.907\\ \hline		
		3 (theoretical) & 5& 82& 1825& 45934& 1240050\\
		3 (numerical) &4.999 &81.973 &1824.235 & 45910.326 & 1239239.988 \\ \hline		
		4 (theoretical) &3 &42 &916 &22974 &620046 \\
		4 (numerical) & 2.999& 41.972& 914.184& 22949.933& 619320.101 \\ \hline			
		5 (theoretical) &3 &42 &916 &22974 &620046 \\
		5 (numerical) & 2.999& 41.973& 914.37& 22958.227& 619599.349\\ \hline		
		6 (theoretical) &3 &42 &916 &22974 &620046 \\
		6 (numerical) & 2.999& 41.973& 914.373& 22958.228& 619599.369\\ \hline		
		7 (theoretical) & 2& 22& 460& 11494& 310044\\
		7 (numerical) & 1.999& 21.978& 459.311& 11471.005& 309269.022\\ \hline		
		8 (theoretical) &3 & 54 & 1215& 30618& 826686\\
		8 (numerical) & 2.999& 53.972& 1214.282& 30600.607& 826296.848\\ \hline		
		9 (theoretical) &2 &28 &610 &153316 &413364 \\
		9 (numerical) & 1.999& 27.987& 609.764& 15311.966& 413302.154\\ \hline		
		10 (theoretical) &1 &14 &305 &7658 &206682 \\
		10 (numerical) &1.000 & 13.991& 304.674& 7647.112& 206330.548\\ \hline		
		11 (theoretical) &1 &8 &155 &3836 &103362 \\
		11 (numerical) & 1.000& 7.983& 154.585& 3825.122& 103074.236\\ \hline	
	\end{tabular}
	\caption{Table of $\mu_1$ moment sequences}
	\label{table:mu-1}
\end{table}}

{\renewcommand{\arraystretch}{1.3}%
\begin{table}[h!]
	\centering
	\begin{tabular}{|c|ccccc|}
		\hline
		Subgroup & $M_1$ & $M_2$ & $M_3$ & $M_4$ & $M_{5}$ \\ \hline \hline
		1 (theoretical) &6& 45& 405& 4131& 45684 \\
		1 (numerical) & 6.000& 44.997& 404.978& 4130.861& 45683.553\\ \hline		
		2 (theoretical) &4& 25& 209& 2083& 22890 \\
		2 (numerical) & 4.000& 24.995& 208.936& 2082.168& 22878.772\\ \hline		
		3 (theoretical) &4& 25& 209& 2083& 22890 \\
		3 (numerical) & 4.000& 24.995& 208.936& 2082.168& 22878.773\\ \hline		
		4 (theoretical) &2& 13 & 105 &1043&  11448 \\
		4 (numerical) & 2.000& 12.995& 104.944& 1042.139& 11436.874\\ \hline		
		5 (theoretical) &3& 15& 111& 1059& 11493 \\
		5 (numerical) & 3.000& 14.995& 110.941& 1058.329& 11485.207\\ \hline		
		6 (theoretical) &3& 15& 111& 1059& 11493 \\
		6 (numerical) & 3.000& 14.995& 110.941& 1058.329& 11485.208\\ \hline		
		7 (theoretical) &2& 9& 59& 539& 5772 \\
		7 (numerical) & 2.000& 8.996& 58.945& 538.250& 5761.490\\ \hline		
		8 (theoretical) &2& 15& 135& 1377& 15228 \\
		8 (numerical) & 2.000& 14.995& 134.937& 1376.264& 15219.817\\ \hline		
		9 (theoretical) &2& 10& 74& 706& 7662 \\
		9 (numerical) & 2.000& 9.998& 73.974& 705.749& 7659.752\\ \hline		
		10 (theoretical) &1& 5& 37& 353& 3831 \\
		10 (numerical) & 1.000& 4.999& 36.977& 352.657& 3826.140\\ \hline		
		11 (theoretical) &1& 4& 22& 186& 1941 \\
		11 (numerical) & 1.000& 3.997& 21.961& 185.552& 1935.804\\ \hline		
	\end{tabular}
	\caption{Table of $\mu_2$ moment sequences}
	\label{table:mu-2}
\end{table}}

{\renewcommand{\arraystretch}{1.3}%
\begin{table}[h!]
	\centering
	\begin{tabular}{|c|cccc|}
		\hline
		Subgroup & $M_2$ & $M_4$ & $M_6$ & $M_8$  \\ \hline \hline
		1 (theoretical) &65& 11076& 2561186& 685324780 \\
		1 (numerical) & 64.995& 11075.733& 2561214.275& 685313387.267 \\ \hline		
		2 (theoretical) & 33& 5540& 1280610& 342662572\\
		2 (numerical) & 32.991 & 5537.488 & 1279721.846& 342293549.829 \\ \hline			
		3 (theoretical) & 33& 5540& 1280610& 342662572\\
		3 (numerical) & 32.991 & 5537.488 & 1279721.927& 342293572.882   \\ \hline		
		4 (theoretical) & 17& 2772& 640322& 171331468\\
		4 (numerical) & 16.991&2769.460 & 639567.852 & 1171083665.383\\ \hline		
		5 (theoretical)  &17& 2772& 640322& 171331468\\
		5 (numerical) & 16.991& 2770.097& 639840.117& 171172626.176  \\ \hline		
		6 (theoretical) & 17& 2772& 640322& 171331468\\
		6 (numerical) & 16.991& 2770.097& 639840.145& 171172628.271 \\ \hline		
		7 (theoretical) & 9& 1388& 320178& 85665916\\
		7 (numerical) & 8.993 & 1385.715 & 319374.334 & 8537815.326 \\ \hline		
		8 (theoretical) & 23& 3696& 853742& 228441640\\
		8 (numerical) & 22.991& 3693.961& 853369.706& 228400307.404 \\ \hline		
		9 (theoretical)  &12& 1850& 426888& 114221002\\
		9 (numerical) & 11.996& 1849.356& 42681.689& 114211031.306 \\ \hline		
		10 (theoretical)  &7& 928& 213454& 57110536 \\
		10 (numerical) & 6.997& 926.945& 213098.809& 56995860.860\\ \hline		
		11 (theoretical)  & 4& 466& 106744& 28555450\\
		11 (numerical) & 3.994& 464.747& 106461.595& 28497167.689\\ \hline		
	\end{tabular}
	\caption{Table of $\mu_3$ moment sequences}
	\label{table:mu-3}
\end{table}}

\hvFloat[%
     floatPos=!htb,
     capWidth=h,% of \columnwidth
     capPos=r,
     objectAngle=90,
     capAngle=90,
     objectPos=l % l c r
]{table}{%
{\renewcommand{\arraystretch}{1.8}
\begin{tabular}
{|c|c|c|c|c|c|}
\hline
{} & Group & Generators & Equation of twist & ${\setstretch{.5}\begin{array}{c}
\text{Minimal field} \\ \text{trivializing the twist}
\end{array}}$ & ${\setstretch{.5}\begin{array}{c}
\text{Normalized} \\ \text{trace distributions}
\end{array}}$ \\ \hline \hline
1 & Trivial & $\{\id\}$ & $ \begin{cases}
x^2+y^2+z^2=0 \\
-2t^2=x^4+y^4+z^4
\end{cases}$ & $\mQ$ & $3_*\phi$\\ \hline
2 & $\mathbb{Z}/2\mathbb{Z}$ & $(1,2)$ &  $y^2 = x^8 - 56x^4 + 16$  & $\mQ(i)$  &  $\frac{1}{2} (\phi) + \frac{1}{2} (3_*\phi)$\\ 
3 & $\mathbb{Z}/2\mathbb{Z}$ & $(1,4)(2,3)$ & $y^2 = -x^8 - 14x^4 - 1$  & $\mQ(i)$  & $\frac{1}{2} (\phi) + \frac{1}{2}(3_*\phi)$ \\ \hline
4 & $\mathbb{Z}/4\mathbb{Z}$ & $(1,2,3,4)$ & ${\scriptscriptstyle y^2 = -x^8 - x^7 - 7x^6 + 7x^5 - 7x^3 -7x^2 + x - 1}$ & $\mQ(\zeta_5)$ & $\frac{3}{4} (\phi) + \frac{1}{4}(3_*\phi)$\\ \hline
5 & $\left( \mathbb{Z}/2\mathbb{Z} \right)^2$ & $(1,2)(3,4)$,$(1,3)(2,4)$ & $y^2=-x^8-56x^4-16$ & $\mQ(\zeta_8)$ & $\frac{3}{4} (\phi) + \frac{1}{4}(3_*\phi)$ \\ 
6 & $\left( \mathbb{Z}/2\mathbb{Z} \right)^2$ & $(1,2)$,$(3,4)$ & $y^2=x^8-14x^4+1$ & $\mQ(\zeta_8)$ & $\frac{3}{4} (\phi) + \frac{1}{4}(3_*\phi)$\\ \hline
7 & $D_4$ & $(1,2,3,4)$,$(1,3)$ & $ 
\begin{aligned}
{\scriptscriptstyle y^2 = } & {\scriptscriptstyle x^8 + 2x^7 - 14x^6 + 14x^5 - } \\ & {\scriptscriptstyle 14x^4 + 14x^3 - 14x^2 + 2x + 1}
\end{aligned}$ & $\mQ(i,\sqrt[4]{3})$ & $\frac{7}{8} (\phi) + \frac{1}{8}(3_*\phi)$\\ \hline
8 & $\mathbb{Z}/3\mathbb{Z}$ & $(1,2,3)$ & $\begin{cases} 
{\scriptscriptstyle x^2+y^2+1=0 }\\
\begin{aligned}
{\scriptscriptstyle 6t^2 = }& {\scriptscriptstyle -23 (x^4+y^4+1) + 16 x - 12 x^2 - 20 x^3  }\\ & {\scriptscriptstyle + 20 y - 12 x y - 12 x^2 y - 16 x^3 y - 12 y^2 } \\ & {\scriptscriptstyle+ 12 x y^2 - 12 x^2 y^2 - 16 y^3 + 20 x y^3}
\end{aligned}
\end{cases}$ & $\mQ(\zeta_9)^+$ & $\frac{1}{3}(3_*\phi) + \frac{2}{3}(\delta_0)$\\ \hline
9 & $S_3$ & $(1,2,3)$, $(1,2)$ & $y^2 = -6x^7 + 21x^4 + 12x$ & $\mQ(\zeta_3,\sqrt[3]{2})$ & ${\scriptscriptstyle\frac{1}{6}(3_* \phi) + \frac{3}{6}(\phi) + \frac{2}{6}(\delta_0)}$\\ \hline
10 & $A_4$ & $(1,2,3)$, $(1,2)(3,4)$ & $\begin{aligned} {\scriptscriptstyle y^2 = }& {\scriptscriptstyle  -x^8 + 4x^7 - 28x^6 + 28x^5 + 14x^4 }\\ & {\scriptscriptstyle +
    28x^3 - 196x^2 + 100x - 61}\end{aligned}$ & ${\setstretch{.6}\begin{array}{c}
\text{Splitting field of} \\ x^4 + 2x^3 + 6x^2 + 6x + 3
\end{array}}$   & ${\scriptscriptstyle\frac{1}{12}(3_* \phi) + \frac{3}{12}(\phi) + \frac{8}{12}(\delta_0)}$\\ \hline
11 & $S_4$ & $(1,2,3,4)$,$(1,2)$ & $\begin{aligned} {\scriptscriptstyle y^2 = }& {\scriptscriptstyle x^8 - 14x^7 + 84x^6 - 294x^5 + 651x^4 } \\ & {\scriptscriptstyle -
    882x^3 + 630x^2 - 126x - 54} \end{aligned}$ & ${\setstretch{.6}\begin{array}{c}
\text{Splitting field of} \\ x^4 - 6x^2 + 2x + 6
\end{array}}$  & ${\scriptscriptstyle\frac{1}{24} (3_*\phi) + \frac{15}{24}(\phi) + \frac{8}{24}(\delta_0)}$ \\ \hline
\end{tabular}}
}%
{Explicit twists realizing all possible Sato-Tate groups}
{Table4}

\clearpage

\bibliographystyle{amsalpha}
\def\bibfont{\small}
\nocite{sutherland2016sato}
\bibliography{AWS.bib} 

\end{document}